\documentclass[
final
]{dmtcs-episciences}

\usepackage[utf8]{inputenc}
\usepackage{subfigure}
\usepackage{latexsym,amssymb,amsfonts,amsmath, amscd, euscript, MnSymbol}

\usepackage[round]{natbib}

\newtheorem{theorem}{Theorem}
\newtheorem{claim}{Claim}
\newtheorem{fact}{Fact}
\newtheorem{problem}{Problem}
\newtheorem{problems}{Problems}
\newtheorem{question}{Question}

\newtheorem{remark}{Remark}

\newtheorem{lemma}[theorem]{Lemma}

\newtheorem{proposition}[theorem]{Proposition}
\newtheorem{corollary}[theorem]{Corollary}
\newtheorem{example}{Example}

\newtheorem{definition}{Definition}

\newcommand{\fac}{Fac}
\newcommand{\age}{Age}

\newcommand{\NN}{{\mathbb N}}
\newcommand{\ZZ}{{\mathbb Z}}

\author{Maurice Pouzet\affiliationmark{1,2}
  \and Imed Zaguia\affiliationmark{3,4}\thanks{Corresponding author. Supported by Canadian Defence Academy Research Program, NSERC and LABEX MILYON (ANR-10-LABX-0070) of Universit\'e de Lyon within the program ''Investissements d'Avenir (ANR-11-IDEX-0007'' operated by the French National Research Agency (ANR) }}

\title[Long paths in graphs]{Graphs containing finite induced paths of  unbounded length}
\affiliation{
Universit\'e Claude-Bernard Lyon 1, CNRS UMR 5208, Institut Camille Jordan, Villeurbanne, France\\
Department of Mathematics and Statistics, University of Calgary, Calgary, Alberta, Canada\\
Department of Mathematics and Computer Science, Royal Military College of Canada, Kingston, Ontario, Canada \\
Department of Mathematics and Statistics, Queen's University, Kingston, Ontario, Canada }


\received{2020-11-18}
\revised{2021-08-05}
\accepted{2021-11-15}
\begin{document}
\publicationdetails{23}{2021}{2}{3}{6915}
\maketitle
\begin{abstract}
  The age $\mathcal{A}(G)$ of a graph $G$ (undirected and without loops) is the collection of finite induced subgraphs of $G$, considered up to isomorphy and ordered by embeddability. It is well-quasi-ordered (wqo) for this order if it contains no infinite antichain. A graph is \emph{path-minimal} if it contains finite induced paths of unbounded length and every induced subgraph $G'$ with this property embeds $G$. We construct $2^{\aleph_0}$ path-minimal graphs whose ages are pairwise incomparable with set inclusion and which are wqo. Our construction is based on uniformly recurrent sequences and lexicographical sums of labelled graphs.
\end{abstract}

%
%



\section{Introduction and presentation of the results}

We consider graphs that are undirected, simple and have no loops. Among those with finite induced paths of unbounded  length, we ask which ones are \emph{unavoidable}. If a graph is an infinite path, it can be avoided in the following sense: it contains a direct sum of finite paths of unbounded length. This latter one, on the other hand, cannot be avoided. Indeed, two direct sums of finite paths of unbounded length embed in each other. Similarly, two complete sums of  finite paths of unbounded length embed in each other. Hence, the direct sum, respectively the complete sum of finite paths of unbounded length are, in our sense, unavoidable. Are there other examples? This question is the motivation behind this article.

We recall that the \emph{age}  of a graph $G$ is the collection  $\age(G)$ of finite induced subgraphs of $G$, considered up to isomorphy and ordered by embeddability (a graph $H$ \emph{embeds} in a graph $H'$ if $H$ is isomorphic to an induced subgraph of $H'$. cf. \cite{fraissetr}).  It is \emph{well-quasi-ordered} (wqo) for this order if it contains no infinite antichain. A \emph{path} is a graph $P$ such that there exists a one-to-one map $f$ from the set $V(P)$ of its vertices into an interval $I$ of the chain $\ZZ$ of integers in such a way that $\{u,v\}$ belongs to $E(P)$, the set of edges of $P$,  if and only if $|f(u)-f(v)|=1$ for every $u,v\in V(P)$.  If $I=\{1,\dots,n\}$, then we denote that path by $P_n$; its \emph{length} is $n-1$, so, if $n=2$, $P_2$ is made of a single edge, whereas if $n=1$, $P_1$ is a single vertex. We denote by $P_{\infty}$ the path on $\NN$. We say that a graph $G$ is \emph{path-minimal} if  it contains induced paths of finite unbounded length and every induced subgraph $G'$ of $G$ with this  property embeds a copy of $G$. Let $\oplus_n P_n$ respectively $\sum_nP_n$ be the direct sum, respectively the complete sum of paths $P_n$ (these operations are a particular case of lexicographical sum defined in Subsection \ref{subsection:lexicographicsum}). These graphs are  path-minimal graphs.  There are others.
Our main result is this.

\begin{theorem}\label{thm:uncoutable ages}
There are $2^{\aleph_0}$ path-minimal graphs whose ages are pairwise incomparable and wqo for embeddability.
\end{theorem}

Our construction uses uniformly recurrent sequences, and in fact Sturmian sequences (or billiard sequences) \cite{morse-hedlund}, and lexicographical sums of labelled graphs.
The existence of $2^{\aleph_0}$ wqo ages is a nontrivial fact. It  was obtained for binary relations in \cite{pouzetetat} and for undirected graphs   in \cite{sobranithesis} and in \cite{sobranietat}. The proofs were based  on  uniformly recurrent sequences. We use similar ideas in \cite{laflamme-pouzet-sauer-zaguia} to prove that there are $2^{\aleph_0}$  linear orders $L$ of order type $\omega$ on $\NN$ such that $L$ is orthogonal to the natural order on $\NN$ (see Theorem 2 and Corollary 2 of \cite{laflamme-pouzet-sauer-zaguia}).

We cannot expect to characterize path-minimal up to isomorphy. Indeed,  recall that two graphs are \emph{equimorphic} if they embed in each other. In general, equimorphic graphs are not isomorphic. Now, if a graph $G$ is path-minimal then as it is easy to observe, every  graph equimorphic to $G$ is path-minimal.

We leave open the following:
\begin{problems}
\begin{enumerate}[{(i)}]
  \item If a graph embeds finite induced paths of unbounded  length, does it embed a path-minimal graph ?
  \item If a graph is path-minimal, is its age wqo ?
  \item If a graph $G$ is path-minimal, how many  path-minimal graphs with the same age are they?
  \item If a graph $G$ is path-minimal, is $G$ can be equipped with an equivalence relation $\equiv$ whose blocks are paths in such a way that $(G, \equiv)$ is path-minimal?
\end{enumerate}
\end{problems}

In some situations, $\oplus_n P_n$ and $\sum_nP_n$  are  the only  path-minimal graphs (up to equimorphy).

\begin{theorem}\label{thm:incomparability graph}
If the incomparability graph of a poset embeds finite induced paths of unbounded length, then it embeds the direct sum or the complete sum of finite induced paths of unbounded length.
\end{theorem}

If  $x,y$ are two vertices of a graph $G$, we denote by $d_G(x,y)$ the length of the shortest path joining $x$ and $y$ if any, and $d_G(x,y):=+ \infty$ otherwise. This defines a distance on the vertex set $V$ of $G$ whose  values belong to  the completion $\overline \NN^+:= \NN^+\cup \{+\infty\}$ of non-negative  integers. This distance is the \emph{graphic distance}. If $A$ is a subset of $V$, the graph $G'$ induced by $G$ on $A$  is an \emph{isometric subgraph} of $G$  if $d_{G'}(x,y)=d_G(x,y)$ for all $x,y\in A$.

If instead of induced paths we consider isometric paths, then

\begin{theorem}\label{thm:isometric}
If a graph embeds isometric finite paths of unbounded length, then it embeds a direct sum of such paths.
\end{theorem}

We examine the primality of the graphs we obtain. Direct and complete sums of finite paths of unbounded length  are not prime  in the sense of   Subsection \ref{subsection:primality}, and not equimorphic to prime graphs.  We construct $2^{\aleph_0}$  examples, none of them being equimorphic to a prime one (Theorem \ref{thm-nonprime-path-minimal}). We construct also $2^{\aleph_0}$ which are prime (Theorem \ref{thm-nonprime-path-minimal}).  These examples are minimal in the sense of \cite{pouzet79}, but not  in the sense of \cite{pouzet-zaguia}.


We conclude this  introduction with:

\subsubsection{An outline of the proof of  Theorem \ref{thm:uncoutable ages}.}

It uses two main  ingredients. One is the so called uniformly recurrent sequences (or words).

Let $L$  be a set, a \emph{uniformly recurrent} word over $L$  with domain  $\NN$ is a  sequence $u:= (u(n))_{n\in \NN}$ of elements of $L$  such that for any given integer $n$ there is some integer $m(u,n)$ such that every factor $v$ of $u$ of length at most $n$  appears as a factor of every factor of $u$ of length at least $m(u,n)$.  \cite{allouche, berthe, lothaire}.   To a uniformly recurrent word $u$ we associate the labelled path $P_{u}$  on $\NN$ where the label of vertex $n$ is the value $u(n)$ of the sequence $u$. If the alphabet $L$ is $\{0,1\}$, we may view $P_{u}$ as a path  with a loop at every vertex $n$ for which $u(n)=1$ and no loop at vertices for which  $u(n)=0$. Next comes the second ingredient.

Fix  a map $\star:  L\times L\rightarrow \{0,1\}$ and denote by  $i\star j$ the image of $(i,j)$. Let  $G_{(u, \star)}$ be the lexicographical sum  over the chain $\omega$ of non-negative integers of copies of the labelled path  $P_{u}$.  This is a labelled graph  whose vertex set is $\NN\times \NN$ with labelling $\ell$. Sets of the form $\{i\}\times \NN$ are the \emph{components} of $G_{(u, \star)}$.  On each component $\{i\}\times \NN$, the graph is a copy of $P_u$, the label $\ell(i,n)$ being $u(n)$,  two vertices  $(i, n)$ and $(j, m)$ of  $G_{(u, \star)}$, such that  $i<j$, are linked by an edge if $\ell(i, n)\star \ell (j, m)=1$.

Since $u$ is uniformly recurrent, the set $\fac(u)$ of finite factors of $u$ is wqo w.r.t the factor ordering hence, by  a theorem of Higman on words \cite {higman},   the ages of $P_{u}$ and  of  $G_{(u, \star)}$ are wqo. Deleting the loops, we get a graph that we denote $\widehat G_{(u, \star)}$ and whose age is also wqo.
Let $A$ be a subset of $V:= V(\widehat G_{(u, \star)})$ such that for each non-negative  integer $i$, the restriction of $\widehat G_{(u, \star)}$ to $A\cap (\{i\}\times \NN)$ is a finite path, and furthermore, the length of these various paths is unbounded.  Then this restriction has the same age as  $\widehat G_{(u, \star)}$ furthermore, two restrictions obtained by this process are equimorphic (i.e. embeddable in each other), thus we do not refer to the family of path and denote by  $\widehat Q_{(u, \star)}$ any of these restrictions. We prove that $\widehat Q_{(u, \star)}$ is path-minimal (Proposition \ref{prop:pathminimal}). If the operation $\star$ is constant and equal to $0$, respectively equal to $1$, $\widehat Q_{(u, \star)}$ is a direct sum, respectively a complete  sum of paths.  To conclude the proof of the theorem,  we need to prove that there is some operation $\star$ and $2^{\aleph_0}$ words  $u$ such that the  ages of $\widehat G_{(u, \star)}$ are incomparable. This is the substantial part of the proof. For that, we prove  in Lemma \ref{lem:longenough} that if $L= \{0,1\}$ and $\star$ is the Boolean sum or a projection  and $u$ is uniformly recurrent then every long enough path in $\widehat G_{(u, \star)}$ is contained in some component. This is a rather technical fact. We think that it holds for  any operation. We deduce that  if $\fac(u)$ and $\fac (u')$ are not equal up to reversal or to addition (mod $2$) of the constant word $1$ the ages of $\widehat G_{(u, \star)}$ and $\widehat G_{(u', \star)}$ are incomparable w.r.t. set inclusion (Proposition \ref{propdistinctages}). To complete the proof of Theorem \ref{thm:uncoutable ages},
we then use the fact that  there are $2^{\aleph_0}$ uniformly recurrent  words $u_\alpha$ on the two letters alphabet $\{0, 1\}$ such that for $\alpha \neq \beta$ the collections $\fac(u_\alpha)$ and $\fac (u_\beta)$ of their finite factors are distinct, and  in fact incomparable with respect to set inclusion (this is a well known fact of symbolic dynamic, e.g. Sturmian words with different slopes will do).

\section{Prerequisites}
We denote by $\omega$ the order type of the chain $(\NN, \leq)$ of non negative integers, by $\omega^d$ the order type of its reverse, by $\zeta$ the order type of $(\ZZ, \leq)$.

The graphs we consider are undirected, simple and have no loops. That is, a {\it graph} is a
pair $G:=(V, E)$, where $E$ is a subset of $[V]^2$, the set of $2$-element subsets of $V$. Elements of $V$ are the {\it vertices} of $G$ and elements of $ E$ its {\it edges}. We denote by $x\sim y$ the fact that the pair $\{x,y\}$ forms an edge. The graph $G$ be given, we denote by $V(G)$ its vertex set and by $E(G)$ its edge set.
If $x,y$ are two vertices of $G$, we denote by $d_G(x,y)$  the length of the shortest path joining $x$ and $y$ if any, and $d_G(x,y):= +\infty$ otherwise.  This defines a distance on $V(G)$ with values in the completion $\overline \NN^+:= \NN^+\cup \{+\infty\}$ of the non-negative integers. This distance is the \emph{graphic distance}. On each connected component, this is an ordinary distance with integer values. If $x$ is any vertex of $G$ and $r\in \NN$, the \emph{ball} of \emph{center} $x$, \emph{radius} $r$,  is the set $B_G(x,r):= \{y\in V(G): d_G(x,y)\leq r\}$. More generally, if $X$ is a subset of $V$, we set $B_G(X, r):= \bigcup_{x\in X} B_G(x,r)$. The \emph{diameter} of $G$, denoted by  $\delta_{G}$, is the supremum of $d_G(x,y)$ for $x,y\in V$. If $A$ is a subset of $V$, the graph $G'$ induced by $G$ on $A$  is an \emph{isometric subgraph} of $G$  if $d_{G'}(x,y)=d_G(x,y)$ for all $x,y\in A$. The supremum of the length of induced finite paths of $G$, denoted by $D_G$,  is sometimes called the \emph{detour} of $G$ \cite{buckley-harary}. We denote by $D_G(x,y)$ the supremum of the lengths of the induced paths joining $x$ to $y$. Evidently, if $x$ and $y$ are connected by a path, $d_G(x,y)\leq D_G(x,y)$. A graph $G$ is \emph{locally finite} if the neighbourhood of every vertex $x$ of $G$ is finite. Equivalently, the balls $B_G(x,r)$ are finite for every vertex $x$ and integer $r \in \NN$. It is easy to see that if a graph $G$ is finite, the  detour of $G$ is finite; if the detour is finite and the graph  is connected then the diameter is bounded; and if the diameter is bounded and $G$ is locally finite then $G$ is finite.
We formally state below the  well-known result of K\"onig \cite{konig} and refer to  \cite{polat} and \cite{watkins} for significant results on infinite paths.

\begin{proposition}\label{konig}If a connected graph is infinite and locally finite then it contains an infinite isometric path
\end{proposition}

We apply to graphs concepts of the theory of relations as developed in \cite{fraissetr}. We feel free to use these concepts for labelled graphs as well without giving more details.   A graph $G$ \emph{embeds} in a graph $G'$ and we set $G\leq G'$  if $G$ is isomorphic to an induced subgraph of $G$.  Two graphs $G$, $G'$ which embed in each other are said to be \emph{equimorphic}. A class $\mathcal C$ of graphs is \emph{hereditary} if $G'\in \mathcal C$ and $G\leq G'$ imply $G\in \mathcal C$. We recall that the \emph{age} of a graph $G$ is the collection $\age(G)$ of finite graphs $H$ which embed in $G$. A collection $\mathcal A$ of finite graphs is the age of a graph if and only if it is hereditary and up-directed (that is any pair $H,H'$ of members of $\mathcal A$ embeds in some member of $\mathcal A$) (see Fra\"{\i}ss\'e 1945 \cite{fraissetr}). An age $\mathcal A$ is \emph{inexhaustible}  if any two of its members embed disjointedly into a third. Equivalently, if   $G$ is a graph with age $\mathcal A$ then for every finite subset $F$ of $V(G)$ the graph induced on $V(G)\setminus F$ has the same age as $G$. Most of the time we consider the members of an age up to isomorphy.

\subsection{Posets, Comparability and incomparability graphs}
Throughout, $P :=(V, \leq)$ denotes an ordered set (poset), that is
a set $V$ equipped with a binary relation $\leq$ on $V$ which is
reflexive, antisymmetric and transitive. We say that two elements $x,y\in V$ are \emph{comparable} if $x\leq y$ or $y\leq x$, otherwise,  we say they are \emph{incomparable}. The \emph{dual} of $P$ denoted $P^{d}$ is the order defined on $V$ as follows: if $x,y\in V$, then $x\leq y$ in $P^{d}$ if and only if $y\leq x$ in $P$. A set of pairwise comparable elements is called a \emph{chain}. On the other hand, a set of pairwise
incomparable elements is called an \emph{antichain}.

The \emph{comparability graph}, respectively the \emph{incomparability graph}, of a poset $P:=(V,\leq)$ is the undirected graph, denoted by $Comp(P)$, respectively $Inc(P)$, with vertex set $V$ and edges the pairs $\{u,v\}$ of comparable distinct vertices (that is, either $u< v$ or $v<u$) respectively incomparable vertices. We set $x\sim y$ if $x$ and $y$ are comparable, and $x\parallel y$ otherwise. A graph $G:= (V, E)$ is a \emph{comparability graph} if the edge set is the set of comparabilities of some order on $V$. From the Compactness Theorem of First Order Logic, it follows that a graph is a comparability graph if and only if every finite induced subgraph is a comparability graph. Hence, the class of comparability graphs is determined by a set of finite obstructions. The complete list of minimal obstructions was determined by Gallai \cite{gallai}.

\subsection{Words, age of words, uniformly recurrent words}

We recall basic facts on words  \cite{allouche,lothaire,pytheas}.

Let $A$ be a finite set, $k$ be  the number of its elements. A \emph{word} $u$ is a sequence of elements of $A$, called \emph{letters},   whose domain is an interval $I$  of $\ZZ$. If  this sequence is finite, the \emph{length} of $u$ is the length $\vert u \vert$ of the sequence. In this  case,  we may  suppose that the domain is $[0, \dots n[$. If $u$ is a word, and $I$ is an interval of its domain, the \emph{restriction} of $u$ to $I$ is denoted by   $u_{\restriction I}$. We denote by $\Box$ the empty word. If  $I$ is infinite, we may suppose that the domain is $\NN$, $\NN^*:= \{0, -1, \dots, -n \dots \}$ or $\ZZ$. If $u$ is  a finite word and $v$ is a word, finite or infinite with domain $\NN$, the \emph{concatenation} of $u$ and $v$ is the word $uv$ obtained by writing $v$ after $u$. If $v$ has domain $\NN^*$, the word $vu$ is similarly defined. A word $v$ is a \emph{factor} of $u$ if $u= u_1vu_2$. This defines an order on the collection of finite words, the \emph{factor ordering}.  The \emph{age} of a word $u$ is the set $\fac (u)$ of all its finite factors endowed with the factor order. We note that a set $\mathcal A$ of finite words is the age of a word $u$ if $\mathcal A$ is an \emph{ideal}  for the factor ordering, that is a non-empty set of finite words which is an \emph{initial segment} (that is if $v\in \mathcal A$, then every factor of $v$ is also in $\mathcal A$) and  is up-directed (that is if $v,v'\in \mathcal A$, then there exists $v''\in \mathcal A$ such that $v$ and $v'$ are factors of $v''$). Note that the domain of $u$ is not necessarily $\NN$, it can be $\NN^*$ or $\ZZ$.
The age of a word $u$ is \emph{inexhaustible} if for every $v, v'\in \fac(u)$ there is some $w$ such that $v'wv'\in \fac (u)$. If an age is inexhaustible this is the age of a word on $\NN$.  A  word $u$  is \emph{recurrent} if every finite factor occurs infinitely often. This amounts to the fact that $\fac(u)$ is inexhaustible.

A word $u$ on $\NN$  is \emph{uniformly recurrent} if for every $n\in \NN$ there exists $m\in \NN$ such that each factor $u(p),...,u(p+n)$ of length $n$ occurs as a factor of every factor of length $m$.

There is a relation between uniformly recurrent words and well-quasi-ordering via the notion of J\'onsson poset.
An ordered set (poset) $P$ is a \emph{J\'onsson} poset if it is infinite and every proper initial segment has a strictly smaller cardinality than $P$. J\'onsson posets were introduced by Oman and Kearnes \cite{kearnes}. Countable J\'onsson posets were studied and described in \cite{pouzetetat, pouzet-sauer, assous-pouzet}. In particular, a countable poset $P$ is J\'onsson if and only if it is well-quasi-ordered, has height  $\omega$, and  for each $n<\omega$, there is  $m<\omega$ such that each element of
height at most $n$ is below every element of height at least $m$.

Their appearance in the domain of symbolic dynamics is due to the following result.
\begin{theorem}\label{unif-recurrent} Let $u$ be a word with domain $\NN$ over a finite alphabet. The following properties are equivalent:
\begin{enumerate}[{(i)}]
\item $u$ is uniformly recurrent;
\item $\fac (u)$ is inexhaustible and wqo;
\item $\fac(u)$ equipped with the factor ordering is a countable J\'onsson poset.
\end{enumerate}
\end{theorem}
The only nontrivial implication is $(ii)\Rightarrow (iii)$ (Lemma II-2.5 of Ages belordonn\'es in \cite {pouzetetat} p. 47). For reader's convenience, we prove it below (the proof   was never published).

\begin{lemma}Let $\mathcal A$ be an inexhaustible age of words. If $\mathcal A$ is not J\'onsson  then it is not wqo.
\end{lemma}
\begin{proof}
Suppose that $\mathcal A$ contains a proper initial segment $\mathcal B$  which is  infinite. Let $\beta_0, \dots \beta_n\dots $ be  a sequence of members of $\mathcal B$ with $\vert \beta_n\vert <\vert \beta_{n+1} \vert$ for every $n\in \NN$. Let $u$ be a word on $\NN$ with $\fac (u)= \mathcal A$. Let $I_n$ be an interval of $\NN$ such that $u_{\restriction I_n}= \beta_n$.

\noindent \textbf{Claim:} For each $n\in \NN$, the set $\mathcal J_n$  of intervals $J $ of $\NN$ such that $I_n\subseteq J$ and $u_{\restriction J}\in \mathcal B$ is finite.\\
\noindent{\bf Proof of Claim.}
If this is not the case, let $[I_n \rightarrow):= \{m\in \NN: \min (I_n)\leq m\}$ and $(\leftarrow I_n[:= \{m\in \NN: m<\min I_n\}$. Then $\fac (u_{\restriction [I_n \rightarrow)}) \subseteq \mathcal B$. Since $\mathcal A$ is inexhaustible $\fac (u_{\restriction (\leftarrow I_n[})= \fac (u)= \mathcal A$. This contradicts $\mathcal B\not = \mathcal A$. This completes the proof of the claim. \hfill $\blacksquare$

For each $n\in \NN$ let  $J_n$ be a maximal element of $\bigcup \mathcal J_n$. The set $I$ of $n\in \NN$ such that $0\in J_n$ is finite, otherwise the union of the $J_n$'s would be $\NN$ and $\mathcal B$ would be equal to $\mathcal A$. For $n \in \NN\setminus I$ set
$x_n := \min J_n -1$, $y_n:= \max J_n+1$, $\overline J_n:= J_n \cup \{x_n, y_n\}$ and $\gamma_n:= u_\restriction {\overline J_n}$.

Since the length of the $\beta_n's$ is increasing, there is an infinite subset $K$ of $\NN \setminus I$ such that $\vert \gamma_n\vert <\vert \gamma_m\vert$ for $n<m$ in $K$. We claim that the set of $\gamma _n$ for $n\in K$ is an infinite antichain.

Indeed, if $\gamma_n\leq \gamma_m$ then either $\vert \gamma_n\vert=\vert  \gamma_m\vert $ in which case  $\gamma_n=\gamma_m$, or $\vert \gamma_n\vert<\vert  \gamma_m\vert $. In this latter case, either $\gamma_{n \restriction J_n \cup \{x_n\}}\leq \gamma_{m \restriction J_m}$ or $\gamma_{n \restriction J_n \cup \{y_n\}}\leq \gamma_{m \restriction J_m}$. Since $\gamma_{m \restriction J_m}\in \mathcal B$ and $\mathcal B$ is an initial segment of $\mathcal A$, either  $\gamma_{n \restriction J_n \cup \{x_n\}}$ or $\gamma_{n \restriction J_n \cup \{y_n\}}$ belongs to $\mathcal B$, contradicting the maximality of $J_n$.
\end{proof}

In the sequel, we suppose that $A:= \{0, \dots k-1\}$. To the word $u$ with domain $\NN$ we may associate the labelled path $P_u$ on $\NN$ where vertex $n$ is labelled by $u(n)$. We may consider $P_u$ as a binary structure made of the path and $k$ unary relations $U_0,...,U_{k-1}$ where $U_i(n)=1$ if $u(n)=i$.  If the alphabet is $\{0,1\}$, this can be expressed in a much simpler form:  $P_u$ is the path where vertex $n$ has a loop if and only if $u(n)=1$.

Due to Theorem \ref{unif-recurrent} and a theorem of Higman, there is a simple correspondence between properties of words and properties of labelled paths.
First we recall Higman's theorem on words.

Let $P:= (V,\leq)$ be an ordered set and let $P^*$ be the set of finite sequences of elements of $V$ that we see as words on the alphabet $V$. The \emph{subword ordering} on $P^{*}$ is defined as follows.  If $u:=u_0,...,u_{n-1}$ and $v:=v_0,...,v_{m-1}$ are two words then we let $u\preceq v$ if there is a 1-1 order preserving map $\varphi$ from the interval $[0,n[$ into the interval $[0,m[$ such that $u_i\leq v_{\varphi(i)}$ for each $i<n$.

\begin{theorem}\label{thm:.higman}\cite{higman}If $P$ is wqo, then $P^{*}$ is wqo.
\end{theorem}

Next, we state the correspondence.

\begin{theorem}\label{unif-recurrent-charact} Let $u$ be a word with domain $\NN$ on a finite alphabet. The following properties are equivalent:

\begin{enumerate}[{(i)}]
\item $u$ is uniformly recurrent;
\item $\age (P_u)$ is inexhaustible and wqo.
\end{enumerate}
\end{theorem}
\begin{proof}
Suppose that $u$ is uniformly recurrent. According to $(i)\Rightarrow (ii)$ of Theorem \ref{unif-recurrent}, $\fac (u)$ is inexhaustible and wqo. Members of $\age (P_u)$ are finite direct sums of labelled paths. They can be viewed as words on the alphabet $\fac (u)$. From Higman's theorem this set of words is wqo. Hence  $\age (P_u)$ is wqo. The fact that $\age (P_u)$ is inexhaustible is obvious.
Conversely, suppose that $\age (P_u)$ is inexhaustible and wqo. Then trivially, $\fac (u)$ is inexhaustible and wqo. Hence, from $(ii)\Rightarrow (i)$ of  Theorem \ref{unif-recurrent}, $u$ is uniformly recurrent.
\end{proof}

As it is well known there are $2^{\aleph_0}$ uniformly recurrent words with distinct ages (e.g. Sturmian words with different slopes, e.g., see Chapter 6 of  \cite{pytheas}).  Despite the fact that  the age of a labelled path $P_u$  does not determine the age of $u$ (it does it  up to reversal), this suffices to prove that:

\begin{lemma}\label{fact1} There are $2^{\aleph_0}$ labelled paths with distinct ages.
\end{lemma}

We will need this fact and more, namely that the graphs $\widehat G_{(u, \star)}$, that we define in Subsection \ref{subsection:lexicographicsum}, have distinct ages (Proposition  \ref{propdistinctages}).  For this,  we start with the following notation and observations. If $v:=v_0\dots v_{n-1}$ is a finite word, let $v^d:=v_{n-1}\cdots v_0$ its reverse;  if $X$ is a set of words, set $X^d:= \{v^d: v\in X\}$. If $\fac(u)$ is  the age of a word $u$ on $\NN$, then  $\fac (u)^d$ is the set of factors of a word on $\NN^*= \{0, -1, \dots, -n \dots \}$ but in general, it is not necessarily  the age of a word on $\NN$. Except if the word $u$ is recurrent and especially if the word is uniformly recurrent. Next, let $\bf 1$ be the constant word equal to $1$; if $v:=v_0\dots v_{n-1}$ is a word on $\{0,1\}$, denote by $v {\dot+}\bf 1$ the word $(v_0 {\dot+}1 \dots  v_{n-1}{\dot+}1)$ where the sum ${\dot+}$ is modulo $2$. If $X$ is a set of words, set $X{\dot+}{\bf{1}}:=  \{v {\dot+} {\bf {1}}: v\in X\}$. Note that if $X$ is the age of a uniformly recurrent word, then $X{\dot+}\bf 1$ is too.

\begin{definition}\label{def:equivalence} Two ages  $\mathcal{A}$ and $\mathcal{A}'$ of uniformly recurrent words are \emph{equivalent} if $\mathcal{A}\in \{\mathcal{A}', \mathcal{A}^{'d},  \mathcal {A'}{\dot+}{\bf 1}, \mathcal {A}^{' d}{\dot+}{\bf 1}\}$.
\end{definition}
This is an equivalence relation and each equivalence class has at most four elements. From  $2^{\aleph_0}$ uniformly recurrent words with distinct ages  we can extract  $2^{\aleph_0}$ ages of words which are not pairwise equivalent. In particular these words  yield $2^{\aleph_0}$ labelled paths with distinct ages. This proves Lemma \ref{fact1} (and more).

\section{A general setting for minimality}

\begin{definition}
Let $\mathcal{C}$ be a class of finite graphs. A graph $G$ is \emph{$\mathcal{C}$-minimal} if $\age(G)$, the age of $G$, contains $\mathcal{C}$ and every induced subgraph $G'$ of $G$ such that $\mathcal{C} \subseteq \age(G')$ embeds $G$. Equivalently, $G$ is minimal among the graphs $G''$ such that $\mathcal{C} \subseteq \age(G'')$,  the collection of these $G''$ being quasi-ordered by embeddability.
\end{definition}

\begin{example} If $\mathfrak{P}$ is the collection of all finite paths, we say that a $\mathfrak{P}$-minimal graph is \emph{path-minimal}. \end{example}

A natural question is:

\begin{question}Given a class $\mathcal{C}$ of finite graphs, does every graph such that $\mathcal{C}\subseteq \age(G)$ embeds some  $\mathcal{C}$-minimal graph?
\end{question}

We have no answer, except in some special cases.

For example, we have:
\begin{fact} \label{fact4} If $G$ is a path-minimal graph and $\age(G)=\downarrow \mathfrak{P}$, the set of graphs which embed in some finite path, then $G$ is a direct sum of finite paths of unbounded  length.
\end{fact}
\begin{proof}Since $\age(G)$ is equal to $\downarrow \mathfrak{P}$, the graph $G$ is the direct sum of paths, some finite and some infinite. If no connected component of $G$ is infinite, then the length  of these connected  components is unbounded and $G$ is equimorphic to a direct sum of finite paths of arbitrarily large length. Else if some connected component of $G$ is infinite it embeds a direct sum of finite paths of unbounded length.
\end{proof}

For an other example,  see Theorem \ref{thm:incomparability graph}.

\begin{question}Let $G$ be a graph such that $\mathcal{C}\subseteq \age(G)$. Does $G$ embed a graph $G'$ such that $\age(G')$ is minimal in the sense that for no induced subgraph $G''$ of $G'$ such that $\mathcal{C}\subseteq \age(G'')$, $\age (G'')$ strictly included in $\age(G')$?
\end{question}

The following are obvious observations.
\begin{enumerate}[$(1)$]
  \item A positive answer to Question 1 yields a positive answer to Question 2.
  \item If $\{G' : G'\leq G\, \rm{ and }\, \mathcal{C}\subseteq \age(G')\}$ is well founded, then the answer to Question 1 is positive.
  \item If the collection of ages of induced subgraphs $G'$ of $G$ containing $\mathcal{C}$ is well founded when ordered by set inclusion, then the answer to Question 2 is positive.
\end{enumerate}

\subsection{Links to other notions of minimality}

\begin{definition}
A graph $G$ is \emph{minimal  for its age} (or simply \emph{minimal}) if it embeds in every induced subgraph $G'$ of $G$ with the same age as $G$ \cite{pouzet79}.
\end{definition}

\begin{fact}\label{fact2} If $G$ is $\mathcal{C}$-minimal, then it is minimal.\end{fact}
\begin{proof}Let $G'$ be an induced subgraph of $G$ such that $\age(G')=\age(G)$. Since $G$ is $\mathcal{C}$-minimal, $\mathcal{C}\subseteq \age(G)$. Since $\age(G')=\age(G)$ we have $\mathcal{C}\subseteq \age(G')$. From the $\mathcal{C}$-minimality of $G$ we deduce that $G$ embeds into $G'$ proving that $G$ is minimal.
\end{proof}

\begin{question} Does there exist a graph such that no induced subgraph with the same age is minimal?(\cite{pouzet79}).
\end{question}

In \cite{pouzet79}, examples of multirelations (not graphs) yield a positive answer to that question.

\begin{fact}\label{fact3} If $G$ is a graph as in Question 3, then for $\mathcal{C}= \age(G)$, $G$ does not embed a $\mathcal{C}$-minimal graph.\end{fact}
\begin{proof} Suppose that some induced subgraph of $G'$ of $G$, with $\mathcal{C}\subseteq \age(G')$, is $\mathcal{C}$-minimal. Then it follows from Fact 2 that $G'$ is minimal. A contradiction.
\end{proof}

\subsection{Indivisibility and minimality}

\begin{definition}An age $\mathcal{A}$ of finite graphs is \emph{indivisible} if for every graph $G$ such that  $\age(G)= \mathcal A$ and every partition of the vertex set of $G$ into two parts $A$ and $B$, one the induced graphs of $G$ on $A$ or $B$ has age $\mathcal{A}$. Equivalently, an age $\mathcal{A}$ of finite graphs is indivisible  if for every $S\in \mathcal{A}$ there is $\overline{S}\in \mathcal{A}$ such that for every partition $A$, $B$ of the vertex set of $\overline{S}$, one of the induced graphs $\overline{S}_{\restriction A}$, $\overline{S}_{\restriction B}$ embeds $S$ (cf.  \cite{pouzet79} p.331 and also \cite{fraissetr}).
\end{definition}

Let $G_n=(V_n,E_n)$ for $n\in \NN$ be a family of graphs having pairwise disjoint vertex sets. We define the \emph{direct sum} of $(G_n)_{n\in \NN}$, denoted $\oplus_n G_n$, is the graph whose vertex set is $\cup_{n\in \NN}V_n$ and edge set $\cup_{n\in \NN}E_n$. The \emph{complete sum} of $(G_n)_{n\in \NN}$, denoted $\sum_n G_n$, is the graph whose vertex set is $\cup_{n\in \NN}V_n$ and edge set $\cup_{i\neq j}\{\{v,v'\} : v\in V_i \wedge v'\in V_j\}\cup  \cup_{n\in \NN}E_n$.

For an  indivisible age $\mathcal A$, graphs that are $\mathcal A$-minimal are easy to describe.

\begin{theorem}If $\mathcal{A}$ is an indivisible age of finite graphs, then every graph $G$ with $\mathcal{A}\subseteq \age (G)$ embeds some graph $G'$ which is $\mathcal{A}$-minimal. In fact, either $G$ embeds a direct sum $\oplus_n G_n$ or a complete sum $\sum_n G_n$ of finite graphs $G_0,G_1,\cdots, G_n,\cdots $ such that $\mathcal{A}=\downarrow \{G_n : n<\omega\}$, the set of graphs which embed in some $G_n$.
\end{theorem}

This result is essentially Lemma 4.2 of \cite{milner-pouzet}. It is based on Propositions IV.2.3.3,  IV.2.3.1 and IV.2.3.2, p.338, 339  of  \cite {pouzet79}. We do not give a proof. We prefer to give the following illustration (Lemma 4.1 of \cite{milner-pouzet}).

\begin{proposition}If a graph embeds  finite cliques of unbounded cardinality, then either it embeds an infinite clique or a direct sum of finite cliques of unbounded cardinality.
\end{proposition}

This proposition is readily obtained from Ramsey's Theorem on quadruples. For the reader's convenience we provide a proof. For each integer $m>0$, choose a complete graph $G_m$ on $m$ vertices and set $V(G_m):=\{v_{0,m},\cdots,v_{i,m},\cdots,v_{m-1,m}\}$. Let $[\NN]^2$ be the set of pairs $\{n,m\}$ with $n<m$ and $f: [\NN]^2 \rightarrow G$ be the map defined by $f(n,m)=v_{n,m}$. Partition the set $[\NN]^4$ of 4-tuples $\{i,j,k,l\}$ into three classes according to whether $f(i,j)$ is equal, adjacent, or not adjacent to $f(k,l)$. From Ramsey's Theorem there exists an infinite subset $X$ of $\NN$ such that $[X]^4$ is in one class. The conclusion follows.

\subsection{Minimality and primality}\label{subsection:primality}

Let $G:=(V,E)$ be a graph. According to Gallai \cite{gallai}, a subset $A$ of $V$ is \emph{autonomous} (also called \emph{module} in the literature) in $G$ if for every $v \not \in A$, either $v$ is adjacent to all vertices of $A$ or $v$ is not adjacent
to any vertex of $A$. Clearly, the empty set, the singletons in $V$
and the whole set $V$ are autonomous in $G$; they are called
\emph{trivial}. An undirected graph is called \emph{indecomposable}
if all its autonomous sets are trivial. With this definition, graphs
on a set of size at most two are indecomposable. Also, there are no
indecomposable graphs  on a three-element set. An indecomposable
graph with more than three elements will be called \emph{prime}.
The graph $P_4$, the path on four vertices, is prime. In fact, as it
is well known, every prime graph contains an induced $P_4$ (Sumner
\cite{sumner} for finite graphs and Kelly \cite{kelly85} for infinite
graphs). Furthermore,  every infinite prime graph contains an
induced countable prime graph \cite{ille}.

\begin{definition}\label{minimality-pou-zag}\cite{pouzet-zaguia} A an infinite graph $G$ is \emph{minimal prime} if $G$ is prime and every prime induced subgraph with the same cardinality as $G$ embeds $G$.
\end{definition}

In \cite{pouzet-zaguia}, the authors provided among other things, the list of minimal prime graphs with no infinite cliques or no infinite independent sets.

\begin{definition}An age $\mathcal{A}$ of graphs is \emph{minimal prime} if $\mathcal{A}$ contains  prime graphs of unbounded cardinality but every proper hereditary subclass contains only finitely many  prime graphs.
\end{definition}

\begin{problem}\cite{oudrar}(Problem 11 page 99 subsection 5.3.1)
If a graph $G$ (more generally a binary relational structure) is minimal prime, is its age minimal prime?
\end{problem}

\section{Path-minimality}

This section is devoted  to the proof  of Theorem  \ref{thm:uncoutable ages}. In Subsection \ref{subsection:lexicographicsum} we define the lexicographical sum of labelled graphs, the set of labels being endowed with an operation $\star$ with $0$-$1$ values. This allows us to construct a path-minimal graph  $\widehat Q_{(u, \star)}$ attached to each uniformly recurrent word $u$ (Proposition \ref{prop:pathminimal}). In Subsection \ref{congruences-path-minimal} we consider path-minimal graphs attached to some periodic sequences and prove that their ages are pairwise incomparable (Proposition \ref{prop:congruences}). In Subsection \ref{generalconstruction} we show first that for some operations $\star$  on $L:= \{0,1\}$  each long enough path in $\widehat Q_{(u, \star)}$  is contained in some component (Lemma \ref{lem:longenough}). Then we prove that there are $2^{\aleph_0}$ such $\widehat Q_{(u, \star)}$'s whose ages are wqo and  incomparable w.r.t. inclusion. This completes the proof of Theorem \ref{thm:uncoutable ages}.

\subsection{Lexicographical sum of labelled graphs}\label{subsection:lexicographicsum}

\begin{definition}A \emph{labelled graph} is a pair $(G,\ell)$ where $G:=(V,E)$ is a graph and $\ell$ is a map from $V$ into a set $L$ of labels. If $H$ denotes a labelled graph, then $\widehat H$ denotes the underlying graph.
In our context, the set $L$ has no order structure, so when we say  that $(G,\ell)$ \emph{embeds into} $(G',\ell')$ and write $(G,\ell) \leq (G', \ell')$, we mean that there is an embedding $h$ of $G$ into $G'$ such that $\ell'(h(v))=\ell(v)$ for all $v\in V$.
\end{definition}

If $L=\{0,1\}$, we may replace a labelling $\ell$ by a loop at every vertex $v$ such that $\ell(v)=1$ and suppress the label $0$.

In the sequel  $L$ is endowed with a map $\star:L\times L\rightarrow \{0,1\}$ and we set $i\star j:= \star (i,j)$. We denote by $\star^d$ the dual of $\star$, defined by $i\star^dj:= j\star i$.

Let $(G_i,\ell_i)_{i\in I}$ be a family of labelled graphs, indexed by a chain $C:=(I,\leq)$,  and suppose that the sets $V(G_i)$, for $i\in I$, are pairwise disjoint. The \emph{lexicographical sum} $\sum_{\star i\in C} (G_i,\ell_i)$ of the family  $(G_i,\ell_i)_{i\in I}$ is the labelled graph $(G,\ell)$ whose vertex set is $\bigcup_{i\in I}V(G_i)$ and edge set
\[E:=(\cup_{i\in I}E_i)\cup \{\{x_i,x_j\} : i<j,  x_i\in V(G_i),  x_j\in V(G_j) \mbox{ and } \ell(x_i)\star \ell(x_j)=1\}.\]
If $\star$ takes only  the value $0$, then $\sum_{\star i\in I} (G_i,\ell_i)$ is the direct sum of $(G_i,\ell_i)_{i\in I}$. On the other hand, if $\star$ takes only the  value $1$, then $\sum_{\star i\in I} (G_i,\ell_i)$ is the complete sum of $(G_i,\ell_i)_{i\in I}$.

If the sets $V(G_i)$'s are not pairwise  disjoint, we replace each $(G_i, \ell_i) $ by a copy on $\{i\} \times V(G_i)$. If all $(G_i, \ell_i)$'s are equal to the same labelled graph $(G, \ell)$ then the lexicographical sum of copies of $(G, \ell)$ indexed by the chain $C$ is the labelled graph
$(G, \ell)\cdot_{\star} C$ on $I\times V(G)$ such $\ell(i,x):= \ell_i(x)$  for each $(i,x)$ and for $i< j$  in $C$, $(i, x)$ and $(j,y)$ form an edge if and only if $\ell(i, x)\star \ell(j, y)=1$ (our notation agrees with the notation of lexicographical sum of chains). In the sequel we will mostly consider lexicographical sum of copies of a labelled graph $(G, \ell)$ indexed by the chain $\omega$ of the non-negative integers. If $G$ is a labelled graph over a set $L$ of labels  and $\star$ denotes the map from $L\times L$ to $\{0, 1\}$, we denote  by $G\cdot_{\star} \omega$ the lexicographical sum of copies of $G$ indexed by the chain $\omega$ of the non-negative integers;  we denote by $\widehat{G\cdot_{\star} \omega}$  the graph obtained by removing the labels of $G\cdot_{\star} \omega$. In the sequel,  our graphs $G$ will be paths labelled by uniformly recurrent sequences.

\begin{figure}[h]
\begin{center}
\includegraphics[width=120pt]{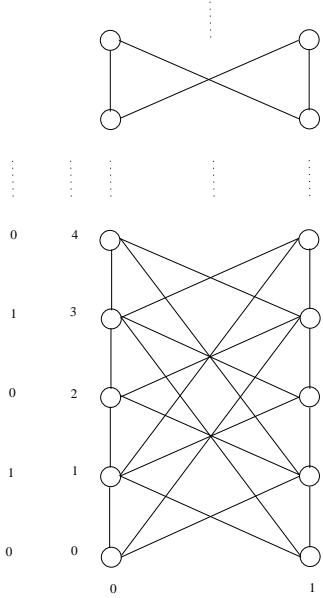}
\end{center}
\caption{Sum of two labeled paths. For $n\in \NN$ set $\ell(2n)=0$ and $\ell(2n+1)=1$. The operation $\star$ is the Boolean sum.}
\label{fig:fig1}
\end{figure}


\begin{remark}\emph{
The fact that our family of labelled graphs is indexed by a chain is enough to serve our purpose. But one could consider a more general notion of lexicographical sum in  a way that the fact that two vertices $x_i$ and $x_j$ are adjacent depends only upon the values of the labels $\ell(x_i)$ and $\ell(x_j)$. This amounts to impose that the sets $V_i$, for $i\in I$, are \emph{Fra\"{\i}ss\'e intervals} (a generalization of autonomous sets;  see \cite{fraisse84} and for a general development, see \cite{ehrenfeucht}) of $(G,\ell)$, that is:
\begin{multline}
\text{for every}\; x_i,x'_i\in V_i, y \not \in V_i,\\
\text{if}\; \ell(x_i)=\ell(x'_i), \text{then}\; \{x_i,y \}\;  \text{is an edge if and only if }\; \{x'_i,y \}\;  \text{is an edge}.
\end{multline}}
\end{remark}
A crucial property of lexicographical sums is the following

\begin{lemma}\label{sumincreasing} Let $(G_i,\ell_i)_{i\in I}$ and  $(G'_{i'},\ell'_{i'})_{i'\in I'}$  be two  families  of labelled graphs, indexed by the  chains $C:=(I,\leq)$ and $C':= (I', \leq)$ respectively. If there is an embedding $\varphi$ from $C$ in $C'$ and for each $i\in I$  an  embedding $f_i$ of $(G_i, \ell_i)$ in $(G'_{\varphi(i)}, \ell'_{\varphi(i)})$ then
$\bigcup_{i\in I}f_i$ is an embedding of $(G, \ell)\cdot_{\star} C$ in $(G', \ell')\cdot_{\star} C'$.
\end{lemma}
The proof is immediate: the fact two vertices $(i, x)$ and $(j,y)$ with $i<j$ form an edge depends only upon the values of $\ell_i(x)\star\ell_j(y)$.

We have also:
\begin{lemma}\label{duality-equality} If $C^d$ is the dual of a chain $C$ then
$(G, \ell)\cdot_{\star} C:= (G, \ell)\cdot_{\star^{d}} C^{d}$.
\end{lemma}

Let  $\mathcal{C}$ be a class of $L$-labelled graphs. Let $\widetilde{\mathcal{C_\star}}$ be the collection of $L$-labelled graphs which are lexicographical sums of $L$-labelled graphs belonging to $\mathcal{C}$ indexed by some chain and let ${\widetilde{\mathcal{C_\star}}}_{<\omega}$ be those that are finite sums.

\begin{theorem}\label{thm:wqo labelled}Let $\mathcal{C}$ be a class of $L$-labelled graphs.
\begin{enumerate}[{$(1)$}]
\item If $\mathcal{C}$ is a hereditary  class of $L$-labelled finite graphs, then ${\widetilde{\mathcal{C_{\star}}}}_{<\omega}$ is hereditary;
\item If $\mathcal{C}$ is an inexhaustible age of $L$-labelled graphs, then ${\widetilde{\mathcal{C_{\star}}}}_{<\omega}$  is an inexhaustible age too;
 \item If $\mathcal{C}$ is wqo then ${\widetilde{\mathcal{C_{\star}}}}_{<\omega}$ is wqo.
 \item If $\mathcal C$ consists of direct sums of labelled paths of bounded length  then the length of labelled paths in ${\widetilde{\mathcal{C_{\star}}}}_{<\omega}$ is bounded.
\end{enumerate}
\end{theorem}
\begin{proof}Only the last two  statements are  nontrivial. They are consequence  of Higman's Theorem \cite{higman}.\\
$(3)$ Let $P$ be the  class $\mathcal{C}$ quasi ordered   by embeddability. To a finite sequence $(G_0,...,G_{n-1})\in P^*$ associate the sum $\sum_{i\in \{1,...,n-1\}}G_i$ and observe that the operation $\sum$ is order preserving from $P^*$ into ${\widetilde{\mathcal{C_{\star}}}}_{<\omega}$: if $(G_0,...,G_{n-1})\preceq (G'_0,...,G'_{n'-1})$ then $\sum_{i\in \{1,...,n-1\}}G_i \leq \sum_{i\in \{1,...,n'-1\}}G'_i$.
Since $P$ is wqo, $P^*$ is wqo by Higman's Theorem. Its  image is ${\widetilde{\mathcal{C_{\star}}}}_{<\omega}$ which is  wqo too.\\
$(4)$ The proof relies on a strengthening of the notion of wqo due to the first author (\cite{pouzet72}, \cite{fraissetr} 13.2.1 page 355).
A collection $\mathcal C$ of labelled  graphs is \emph{$1$-well} if the collection $\mathcal C_{[\bf 1]}$ made of pairs $(G,U)$ where $G\in \mathcal C$ and $U$ is a unary relation on $V(G)$ is wqo. A collection of paths (labelled or not) of unbounded length is not $1$-well. Indeed, to each path $G$ associate $(G, U)$ where $U$ take value $1$ on the end-vertices of $G$ and $0$ on the other vertices. Then paths with different length are incomparable. Now, suppose that  $\mathcal C$ consists of direct sums of labelled paths of bounded length then, with  Higman's Theorem,  $\mathcal C_{[\bf 1]}$ is wqo, and again by Higman's  Theorem  $({\widetilde{\mathcal{C_{\star}}}}_{<\omega})_ {[\bf 1]}$  is wqo. Hence the length of paths in ${\widetilde{\mathcal{C_{\star}}}}_{<\omega}$ is bounded.
\end{proof}

\begin{lemma}\label{lem:sameage} Let $G$ be a labelled graph and let $G\cdot_{\star} C$ be the  lexicographic sum of copies of $G$ indexed by the chain $C:= (I, \leq)$. Let $J$ be an infinite subset of $I$, $(A_j)_{j\in J}$ be a family of subsets of $V(G)$, $\mathcal S:= (G_j)_{j\in J}$ where $G_j:= G_{\restriction A_j}$ and   $G_{(\star, \mathcal S)}$ be the graph  induced by  $G\cdot_{\star} C$ on  $\bigcup_{j\in J} \{j\}\times A_j$. Then

\begin{enumerate}[$(1)$]
\item  $\age (G_{(\star, \mathcal S)})= \age (G\cdot_{\star} C)$ if $\age (G)= \bigcup_{i\in J} \age (G_i)$.

\item Furthermore, if $C_{\restriction J}$ has order type $\omega$ or $\omega^d$ and  $(A_j)_{j\in J}$,  $(A'_j)_{j\in J}$ are two families  of finite subsets of $V(G)$ such that $\age (G)= \bigcup_{i\in J} \age (G_i)= \bigcup_{i\in J} \age (G'_i)$ then $G_{(\star, \mathcal S)}$ and  $G_{(\star, \mathcal S')}$ are equimorphic.

\end{enumerate}
\end{lemma}
\begin{proof} $(1)$ Let $H\in \age (G\cdot_{\star} C)$. Let $F \subseteq I\times V(G)$ such that $(G\cdot_{\star} C)_{\restriction F}$ is isomorphic to $H$. Let $K$ be the projection of $F$ on $I$. For $k\in K$, set $F_k:= \{x\in V(G): (k, x)\in F\}$. Since $\age (G)= \bigcup_{i\in J} \age (G_i)$,  $G_{\restriction F_k}$ embeds in some $G_j$. Write $k_0, \dots,  k_{m-1}$ the elements of $K$ in an increasing order, so $k_i< k_{i+1}$ for $i\in \{0,...,m-2\}$.  For each $i$, $0\leq i<m$,  we may select  $n_{k_i}$ such that $G_{\restriction F_{k_i}}$ embeds in $G_{n_{k_i}}$, this $G_{n_{k_i}}$ being isomorphic to $G_{\restriction \{n_{k_i}\}\times A_{n_{k_i}}}$ and the indices $n_{k_i}$ strictly increasing with $i$. Then, with Lemma \ref{sumincreasing},  one can see that $H$ embeds into $(G\cdot_{\star} C)_{\restriction \bigcup_{0\leq i <m} \{n_{k_i}\}\times A_{n_{k_i}}}$, hence in $G_{(\star,  \mathcal S)}$, thus $H\in \age(G_{(\star, \mathcal S)})$. This proves our assertion.

$(2)$We prove that $G_{(\star, \mathcal S)}$ embeds in $G_{(\star, \mathcal S')}$.  We may suppose that $C_{\restriction J}$ has order type $\omega$ and identify it with $\NN$.  We built  a strictly increasing self map $\varphi$ on $\NN$  such that $(G\cdot_{\star} C_{\restriction J})_{\restriction \{n\}\times A_n}$ embeds in $(G\cdot_{\star} \omega)_{\restriction \{\varphi(n)\}\times A'_{\varphi(n)}}$. The conclusion follows with Lemma \ref{sumincreasing}.
\end{proof}
\begin{corollary}\label{cor:zeta-omega} $G\cdot_{\star} \zeta$, $G\cdot_{\star} \omega^d$ and $G\cdot_{\star} \omega$ have the same age.
\end{corollary}

\begin{definition}
To a uniformly recurrent word $u$ on $\NN$ over the alphabet $L$ we associate four graphs $G_{(u, \star)}$, $Q_{(u, \star)}$, $\widehat G_{(u, \star)}$ and $\widehat Q_{(u, \star)}$. The graph  $G_{(u, \star)}:= P_{u}\cdot_{\star} \omega$ is the lexicographical sum of $\omega$ copies of the labelled graph $P_u$. To emphasize, the vertex set $V(G_{(u, \star)})$ is $\NN\times \NN$; on each component $\{i\} \times \NN$, the graph induced by $G_{(u, \star)}$ is a copy of the path $P_u$ labelled by $u$; the label of vertex $(i, x)$ is $u(x)$,   two vertices  $(i, x)$ and $(j, y)$ of  $G_{(u, \star)}$ with $i<j$ are linked by an edge if and only if $u(x)\star u(y)=1$. Let $(I_n)_{n\in \NN}$ be a family of intervals of $\NN$ of unbounded finite length then $Q_{(u, \star)}$ is  the induced subgraph of $G_{(u, \star)}$ on  $I:=\bigcup_{n\in \NN}\{n\}\times I_n$.
Let $\widehat G_{(u, \star)}$ and $\widehat Q_{(u, \star)}$ be the graphs obtained by removing the labels of $G_{(u, \star)}$ and $Q_{(u, \star)}$ respectively. Hence, $\widehat Q_{(u, \star)}$ is the restriction to $I$ of $\widehat G_{(u, \star)}$.
\end{definition}

\begin{proposition}\label{prop:pathminimal}The graphs $\widehat G_{(u, \star)}$ and $\widehat Q_{(u, \star)}$ have the same age and this age is wqo. Furthermore $\widehat Q_{(u, \star)}$ is path-minimal.
\end{proposition}
\begin{proof}
First, we prove that   $\widehat G_{(u, \star)}$ and $\widehat Q_{(u, \star)}$ have the same age.
This follows from the fact that $G_{(u, \star)}$ and $Q_{(u, \star)}$ have the same age. This fact is an immediate consequence of Lemma \ref{lem:sameage}.

Next we prove that $\age(\widehat G_{(u, \star)})$,  the age of $\widehat G_{(u, \star)}$,  is wqo. Again, this fact follows from the fact that  $\age(G_{(u, \star)})$ is wqo. This last fact follows from Item $(3)$ of Theorem \ref{thm:wqo labelled}. Indeed, according to Theorem \ref{unif-recurrent-charact}, $\age(P_u)$ is wqo. If $\mathcal C:= \age (P_u)$ then ${\widetilde{\mathcal{C_{\star}}}}_{<\omega}$ is wqo by Item $(3)$ of Theorem \ref{thm:wqo labelled}. Since $\age(G_{(u, \star)})\subseteq {\widetilde{\mathcal{C_{\star}}}}_{<\omega}$,  $\age(G_{(u, \star)})$ is wqo.

We prove that any other family of intervals of finite  unbounded length yields an equimorphic graph to $\widehat Q_{(u, \star)}$ and furthermore  that  this graph as well as $\widehat Q_{(u, \star)}$ is path-minimal. For that, we apply the  following claim.


\noindent \textbf{Claim:} Let $G:= G_{(u, \star)}$, $Q:=Q_{(u, \star)}$ and $\widehat G$, $\widehat Q$ be the graphs obtained by deleting the labels.
\begin{enumerate}[$(1)$]
  \item $\widehat Q$ contains  finite induced paths of unbounded length.
  \item Let  $A\subseteq V(G)$. If  $\widehat G_{\restriction A}$ contains finite paths of unbounded length, then $\widehat Q$ embeds into  $\widehat G_{\restriction A}$.
\end{enumerate}

\noindent \textbf{Proof of Claim:} For $n\in \NN$ let $A_n:=A\cap (\{n\}\times P_u)$. We claim that for each integer $m$ there exists some $n_m$ such that $\widehat G_{\restriction A_{n_m}}$ contains a path of length at least $m$.  Suppose that  this is not the case. This means that there is   some integer $k$  such that the length of paths in each $A_n$ is bounded by $k$. In this case, $\widehat G_{\restriction A_{n}}$ decomposes into a direct sum of labelled paths of length at most $k$. According to $(4)$ of Theorem \ref{thm:wqo labelled}, the length of paths in $G_{\restriction A}$ is bounded, contradicting our hypothesis. This proves our claim. For each $m$,  let $I_{n_m}$ be an interval of $\NN$ of length $m$ such that $\{n_m\}\times I_{n_m}$ is included into some $A_{n_m}$, the   $A_{n_ m}$ being distinct and in an increasing order. The set $S:= \{P_u\restriction I_{m_n}: m\in \NN\}$ is an infinite subset of $\age(P_u)$. Since $u$ is uniformly recurrent, it follows from Theorem \ref{unif-recurrent-charact} that $S$ is cofinal in $\age (P_u)$, that is every element of $\age (P_u)$ embeds in some element of $S$.  Thus Lemma \ref{lem:sameage} applies,   $Q$ embeds into $G_{\restriction A}$, hence $\widehat Q$ embeds into $\widehat G_{\restriction A}$. This completes the proof of the claim.

The proof of the Proposition \ref{prop:pathminimal} is now complete.
\end{proof}

\subsection{Congruences and path-minimal graphs}\label{congruences-path-minimal}
In this subsection, we start with a simple  construction. It gives  an   illustration of our techniques and allows us to show that there are countably many  path-minimal graphs with incomparable ages (Proposition \ref{prop:congruences}).

%

 Let $k\geq 2$ be an integer and $\equiv_k$ be the congruence modulo $k$ on the integers. Let $\widehat G_{k}$  be the graph with vertex set $\NN\times \NN$ such that two  vertices $(x,y)$ and $(x',y')$ form  an edge if either $x=x'$ and $|y-y'|=1$ or $x\neq x'$ and $y \not \equiv_k y'$.   Let $\widehat Q_k$ be the restriction of $\widehat {G_k}$ to $\{(n,m) : m< n+5\}$.

 The graph $\widehat G_k$ can be obtained by the process described in Subsection \ref{subsection:lexicographicsum}. Indeed, let $u:=(u(n))_{n\in \NN}$ where $u(n)$ is the residue of $n$ modulo $k$. Hence $u$ is a periodic word on  the alphabet $\{0,...,k-1\}$. Define $\star$ from $\{0, \dots, k-1\}\times \{0, \dots, k-1\}$ into $\{0,1\}$ by setting  $i\star i=0$ for $i=0, \dots, k-1$ and $i\star j=1$ otherwise. To the word $u$ we associate the labelled path $P_u$ on $\NN$ where vertex $n$ is labelled by $u(n)$.  Let $G$ be the  lexicographical sum  of copies of $P_u$ indexed by the chain $\omega$, say $G:=P_u \cdot_{\star}\omega$, and $\widehat G$ be the graph obtained by removing the labels of $G$.   Then $\widehat G_k = \widehat G$. Note that  $\widehat G_2=\widehat G_{(u, \star)}$ where $u$ is the periodic sequence $010101 \dots$ and $\star$ is the Boolean sum on $\{0,1\}$.

\begin{proposition}\label{prop:congruences}
For each integer $k\geq 2 $, $\widehat Q_k$ is a path-minimal graph and for different values of $k$ the ages of $\widehat Q_k$ are incomparable w.r.t. inclusion.
\end{proposition}

 The fact that $\widehat Q_k$ is path-minimal follows from Proposition \ref{prop:pathminimal}. We give a self-contained proof of the full proposition above avoiding the use of wqo.  It  relies on  claims 2, 3 and 4 stated below.

For $i\in \NN$, let $\NN_i:= \{(i, j) : j\in \NN\}$ and let $\widehat G_k(i)$ be the graph induced by $\widehat G_k$ on $\NN_i$. This is an infinite path that  we  call  a \emph{component} of $\widehat G_k$. The restriction  $\widehat Q_k$ of $\widehat G_k$ to the set $X_i:= \{(i,j) : j< i+5\}$ is an induced path of length $i+5$  that we call also a component of $\widehat Q_k$.
We note that $\widehat Q_k$ is a union of finite paths of unbounded length.

\noindent \textbf{Claim 1:} $\delta(\widehat G_2)=3$  and $\delta(\widehat G_k)=2$ for $k\geq 3$.

\begin{proof} Let $a:=(x,y)$ and $a':=(x',y')$ be two distinct vertices. We claim that their distance is at most $3$ if $x\not = x'$ or $k>2$, and at most $3$ if $x=x'$ and $k=2$. Suppose that $x\not =x'$, eg.,  $x<x'$. If $y' =y$ then   $(x,y)\sim (x',y+1) \sim (x', y)$ and the distance is $2$. Suppose that $y\not =y'$, e.g., $y<y'$. If $y'\not \equiv_k y$ then $(x,y)\sim (x',y')$ and the distance is $1$. If  $y' \equiv_k y$ then $(x,y)\sim (x',y+1) \sim (x', y')$ and the distance is $2$. Suppose that $x=x'$. We may suppose that $y\not =y'$  otherwise the distance is $0$. We may thus also suppose $y<y'$. Let $r$ be the residue of $y'-y$ mod $k$, i.e. $y'=my+r$ with $0\leq r<k$. If $r=0$ then $(x,y)\sim (x+1,y+1) \sim (x',y')$ and thus $d_{\widehat G_k}(a,a')\leq 2$.  If $k=2$ and $r=1$ then either $m=1$ and  thus $(x,y)\sim (x, y')$, and the distance is $1$, or $m>1$,  in which case  $(x,y)\sim (x+1,y+1) \sim (x, y'-1)\sim (x,y')$ and hence $d_{\widehat Q_k}(a,a')\leq 3$. If $k>2$ let $r':= 2$ if $r=1$ and $r':= 1$ if $r>1$. Then  $(x, y) \sim (x+1,y+r') \sim (x,y')$ and hence $d_{\widehat G_k}(a,a')\leq 2$.
\end{proof}

\noindent \textbf{Claim 2:} An induced  path of length at least six of $\widehat G_k$ is necessarily  included  in a component of $\widehat G_k$.

\begin{proof} Let $R:=((x_i,y_i))_{i\geq 0}$ be an induced path in $\widehat G_k$ of length at least four. Then the following four statements below hold:
\begin{enumerate}[$(1)$]
 \item For all $i$ and for all $j \geq i+2$ we have $x_i=x_{j}$ or  $x_i=x_{j+1}$.\\
\emph{Proof:} Suppose  $x_i \neq x_{j}$ and $x_i \neq x_{j+1}$.  Since $(x_i,y_i) \not \sim (x_{j}, y_{j})$ and $(x_i,y_i) \not \sim (x_{j+1}, y_{j+1})$,  $y_{j}  \equiv_k y_i  \equiv_k  y_{j+1} $. Since $(x_{j}, y_{j})\sim(x_{j+1}, y_{j+1})$ and $y_j  \equiv_k  y_{j+1}$ we must have $x_{j}= x_{j+1}$ and  $|y_{j}- y_{j+1}|=1$ which is impossible since $k\geq 2$.$\blacksquare$
  \item For all $i$ and for all $j \geq i+2$, $x_i$ is equal to at least three elements among $x_j,x_{j+1},x_{j+2},x_{j+3}$.\\
\emph{Proof:} Suppose for a contradiction that $j\leq r<s\leq j+3$ such that $x_i\not \in  \{x_r,x_s\}$. Since $R$ is induced we have $y_r\equiv_ky_i \equiv_k y_s$. It follows from (1) that $s-r\neq 1$.  There are three cases to consider.

 \noindent \textbf{Case 1:} $r=j$ and $s=j+2$.\\
It follows from (1) that $x_{j+1}=x_i=x_{j+3}$. Since $x_{j+3}\neq x_{j}$ and $R$ is induced we infer that $y_{j+3}\equiv_k y_{j}$. But then $y_{j+3}\equiv_k y_{j+2}$ contradicting $x_{j+3}\neq x_{j+2}$.

 \noindent \textbf{Case 2:} $r=j+1$ and $s=j+3$.\\
This follows  from Case 1 and symmetry.

 \noindent \textbf{Case 3:} $r=j$ and $s=j+3$.\\
It follows from (1) that $x_{j+1}=x_i=x_{j+2}$ and hence $|y_{j+1}-y_{j+2}|=1$. Furthermore, $y_{j+3}\equiv_k y_{j+1}$ and $y_j\equiv_k y_{j+2}$. Hence, $y_{j+1}\equiv_ky_i\equiv_k y_{j+2}$ contradicting $|y_{j+1}-y_{j+2}|=1$.$\blacksquare$

  \item If $R$ has length at least five, there exists at most one index $i\geq 2$ such that $x_0\neq  x_i$.\\
\emph{Proof:} Suppose $2\leq i<j$ such that $x_i\neq x_0\neq x_j$. It follows from (1) that $x_{i+1}=x_0$. Then $j\neq i+1$. It follows from (2) that we may assume without loss of generality that $j=i+2$. Then $(x_{i+3},y_{i+3})\in R$  or $i>2$ and hence $(x_{i-1},y_{i-1})$ is not a neighbour of $(x_{0},y_{0})\in R$. We obtain a contradiction by applying (2) to the sequence $(x_0,y_0)$ and $(x_i,y_i),(x_{i+1},y_{i+1}),(x_{i+2},y_{i+2}),(x_{i+3},y_{i+3})$ or to the sequence $(x_0,y_0)$ and $(x_{i-1},y_{i-1}),(x_{i},y_{i}),(x_{i+1},y_{i+1}),(x_{i+2},y_{i+2})$.$\blacksquare$

\item Let $(x_0,y_0),(x_1,y_1),...,(x_6,y_6)$ be an induced path in $\widehat G_k$ for $k\geq 2$. If there exists $i$ such that $x_0\neq x_{i}$, then $i=1$.\\
\emph{Proof:} If not, then by (3) there is a unique $i\geq 2$ such that $x_0\neq x_i$ . If $i\neq 4$, then there are then at least two indices $r,s\geq 2$ such that $|r-s|=1$ and $(x_r,y_r)$ and $(x_s,y_s)$ are not adjacent to $(x_{i},y_{i})$. Hence, $y_r\equiv_k y_0\equiv_k y_s$. Since $x_r=x_0=x_s$ we infer that $|y_s-y_r|=1$ contradicting $y_s \equiv y_r \bmod k$. We now assume $i=4$. Hence, $x_0=x_2=x_3=x_5=x_6$. We now consider the path $(x'_0,y'_0)=(x_6,y_6), (x'_1,y'_1)=(x_5,y_5),...,(x'_5,y'_5)=(x_1,y_1),(x'_6,y'_6)=(x_0,y_0)$. Since $x'_0=x'_4$ we apply the previous argument (as in the case $i_0\neq 4$) to this path and we obtain a contradiction.$\blacksquare$

  \item Let $(x_0,y_0),(x_1,y_1),...,(x_6,y_6)$ be an induced path in $\widehat G_k$ for $k\geq 2$. Then $x_0=x_1=...=x_6$.\\
\emph{Proof:} It follows from (4) that $x_0=x_2=...=x_6$. By symmetry, $x_6=x_4=x_3=x_2=x_1=x_0$. Therefore, $x_0=x_1=...=x_6$ as required.$\blacksquare$
\end{enumerate}
\end{proof}

The value six is best possible:  The sequence  $(0,0),(1,1),(0,k),(0,k+1),(1,2k),(0,2k+1)$  yields an induced path in $\widehat G_k$  which is not included in a  component of $\widehat G_k$.


%
\noindent \textbf{Claim 3:} The graph $\widehat Q_k$ is path-minimal.

\begin{proof}
Let $Y$ be a subset of $V(\widehat Q_k)$ such that the graph induced by  ${\widehat Q_k}$ on $Y$ contains finite induced paths of unbounded  length. From Claim 2 above, each induced path  of length  at least six is included in some component of $\widehat G_k$, hence in a component  of $\widehat Q_k$.  Since the components of $\widehat Q_k$ are finite, there is an infinite sequence of integers $(i_n)_{n\in \NN}$ such that $\widehat Q_k {\restriction  Y\cap X_{i_n}} $  contains a path on a subset $Z_{i_n} $ of length $n$. According to Lemma \ref{lem:sameage},   $\widehat Q_k$ embeds in $\widehat Q_k{\restriction \bigcup_{i_n\in \NN} Z_{i_n}}$.
\end{proof}

\noindent \textbf{Claim 4:} For each $k\geq 2$, $\age(\widehat Q_k)= \age(\widehat G_k)$ and if $k\neq k'$, then $\age(\widehat G_k)$ and $\age(\widehat G_{k'})$ are incomparable with respect to set inclusion.

\begin{proof}The first part is obvious. For the second part, let $R$ be a path of length $m\geq \max \{k+1, 6\}$ in a component of $\widehat G_k$.  Let  $x$ be a vertex in another component.  The trace on $R$ of the neighborhood of $x$ is a union of paths, each of length $k-1$;  two such paths consecutive in $R$ being separated by a single vertex. Then  $\widehat G_k{\restriction R\cup \{x\}}$ cannot be embedded in $\widehat G_{k'}$. Indeed, otherwise, by Claim 2, the path $R$ of length $m$ will be mapped in a component of $\widehat G_{k'}$ on a path $R'$ and the element $x$ will be mapped on an element $x'$  in another component. Hence  $\widehat G_k{\restriction R\cup \{x\}}$ will be  isomorphic to $\widehat G_{k'}{\restriction R'\cup \{x'\}}$. But the trace on  $R'$ of the neighborhood of $x'$ is a union of paths, each of length $k'-1$ and   two such paths consecutive in $R'$ are separated by a single vertex. Since $k\not =k'$ this is not possible.
\end{proof}

Sequences other  than  $u:=(u(n))_{n\in \NN}$, where $u(n)$ is the residue of $n$ modulo $k$, will produce path-minimal graphs with different ages. This is the subject of the next two subsections.

\subsection{A general construction of path-minimal graphs with incomparable ages}\label{generalconstruction}

Our aim in this subsection  is to show that there is an operation $\star$ on $L:= \{0, 1\}$ and  $2^{\aleph_0}$ uniformly recurrent words in such a way that the ages of the graphs $\widehat Q_{(u,\star)}$ given in Proposition \ref{prop:pathminimal} are pairwise incomparable w.r.t. inclusion. Since $\widehat Q_{(u,\star)}$ and $\widehat G_{(u,\star)}$ have the same age, we only need to  consider $\widehat G_{(u,\star)}$.
We restrict ourselves to $L:= \{0, 1\}$ while several arguments extend to general alphabets. We leave the general case to further investigations.

The proof has two parts, developed in the next two subsections.  We show first that long paths in $\widehat G_{(u,\star)}$ are contained  in components (Lemma \ref{lem:longenough}). Next, we use that fact to prove that well-chosen words provide distinct ages (Proposition  \ref {propdistinctages}).

\subsubsection{Long paths are contained  in components}

Let $L:=\{0,1\}$. The \emph{first projection}  is the operation  defined by $i\star j=i$ for all $i,j$, while the \emph{second projection} is defined by $i\star j=j$ for all $i,j$.
The \emph{Boolean sum} is the operation $\dot +$ on $\{0, 1\}$ such that $0{\dot +} 0= 1{\dot +}1=0$ and $0{ \dot + }1=1{\dot +}0=1$. This operation is also the sum modulo two (cf. Section \ref{congruences-path-minimal}). Note that the term Boolean Sum is also used for the inclusive \emph{or} in Boolean Algebra in which case $1+1=1$. If $\star$ is an operation on $\{0,1\}$, its \emph{dual}, $\star^d$,  is the operation defined by $i\star^dj:= j\star i$. Let  $\overline \star$ be the operation on $\{0,1\}$ defined by $i\overline \star j:= (i {\dot+}1)\star (j{\dot +}1)$. The Boolean sum is invariant under the two transformations defined above (i.e. $(\dot +)^d=\dot +$ and $\overline {\dot +}=\dot +$). The dual of the  first projection is the second, while  $i\overline \star j= i{\dot +} 1$ if $\star$ is the first projection and $ i\overline \star j= j{\dot +} 1$ if $\star$ is the second.

Let $u:= (u_n)_{n\in \NN}$ be a $0$-$1$ sequence; let $u{\dot+} \bf{1}$ be the sequence defined by  $u{\dot+}{\bf 1}(n):= u(n) {\dot+}1$.
\begin{lemma} Let $u$  be a $0$-$1$ sequence. Then $\widehat G_{(u, \star)}= \widehat G_{(u{\dot +}{\bf 1}, \overline \star)}$ and  $\age (\widehat G_{(u, \star)})=\age (\widehat G_{(u, \star^d)})$.
\end{lemma}
\begin{proof}
The verification of the first equality is immediate. For the second, observe that from Corollary \ref{cor:zeta-omega},  $\age (\widehat G_{(u, \star)}):= \age (\widehat{P_{u} \cdot_{\star}\omega})=\age (\widehat{P_{u} \cdot _{\star}\omega^d})$. Since by Lemma \ref{duality-equality},   $\widehat {P_{u} \cdot _{\star}\omega^d}= \widehat {P_{u}\cdot_{\star^d}\omega}=\widehat G_{(u, \star^d)}$, the second equality follows.
\end{proof}

This fact suggests that we consider operation $\star$ and $\star'$ as equivalent for each $\star'\in \{\star, \star^d, \overline \star, (\overline \star)^d\}$.

 Let $u$ be a  uniformly recurrent and non constant word on $L$.  For $i\in \{0,1\}$  there is a maximum, denoted by  $l_i(u)$, of the  length of factors of $i's$ in $u$.  Let $l(u):= \max \{l_0(u), l_1(u)\}$ and  set $\varphi(u):= 2(l(u)+1)+1$ if $u$ is not the periodic word $010101\dots$; otherwise set $\varphi(u):=6$. Note that $\varphi(u)= \varphi(u\dot{+}1)$.

\begin{figure}[h]
\begin{center}
\includegraphics[width=150pt]{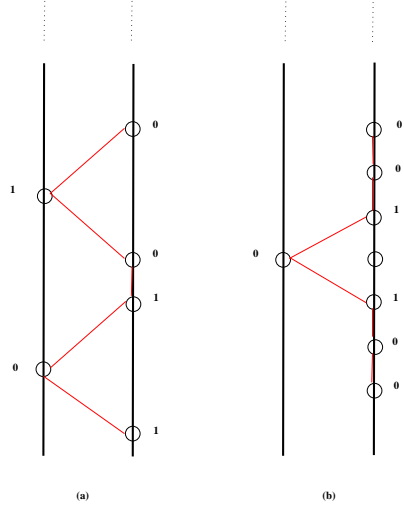}
\end{center}
\caption{Two examples of paths not contained in one component of  $\widehat G_{(u, \star)}$ if $\star$ is the Boolean sum. The thick vertical lines represent components of $\widehat G_{(u, \star)}$.}
\end{figure}


\begin{lemma}\label{lem:longenough}Let $\star$ be  the Boolean sum or an operation equivalent to the second projection. If $u$ is any  uniformly recurrent and non constant word, then every  path of length at least  $\varphi(u)$ in $\widehat G_{(u, \star)}$ is contained  in one component.
\end{lemma}
\begin{proof} Set $G:=\widehat G_{(u, \star)}$ and for $i\in \NN$, set $X_i:= \{i\}\times \NN$. Throughout this proof, $R$ denotes a finite induced path of $G$, $V(R)$ denotes its set of vertices and $E(R)$ its set of edges. We let $V_0(R)$, respectively $V_1(R)$,  denote the vertices of $R$ labelled $0$, respectively labelled 1, in $G_{(u, \star)}$.

\begin{enumerate}[$(\bullet )$]

\item  $\star$ is the Boolean sum.
We observe first  that  since $\widehat G_{(u, \star)}=\widehat G_{(u\dot{+}1, \star)}$, $R$ is a path of $\widehat G_{(u\dot{+}1, \star)}$. Since $\varphi(u)= \varphi(u\dot{+}1)$, we may suppose that the values of $u$ are defined up to complementation. Also, and since $R$ is connected $V_0(R)\neq \varnothing \neq V_1(R)$.

\noindent \textbf{Claim 1:} Every induced path, regardless of its length, is contained in at most three components of $G$.\\
\noindent \emph{Proof of Claim $1$:} If $R$ is contained in at most two components, then we are done. Next, suppose that $R$ is contained in at least three components of $G$.\\
\noindent \textbf{Sub Claim:} Up to the complement of $u$, there are three vertices $a,b,c$ of $R$ contained in three distinct components of $G$ such that $\ell(a)=1$,  $\ell(b)=0$ and $\ell(c)=0$.\\
\emph{Proof of Sub Claim:} Let $X$ be a component meeting $R$ in a vertex $a$ labeled $1$. If all vertices of $R$ not in $X$ have the same label $0$ then pick $b$ and $c$ not in $X$ and in two distinct components. If all vertices of $R$ not in $X$ have the same label $1$, then since $R$ is connected there must be a vertex $a'$ in $X$ labeled $0$, then choose $a'$ and pick $b$ and $c$ not in $X$ and in two distinct components.  Replacing $u$ by its complement $u\dot{+}1$, the vertices $a', b, c$ satisfy the conclusion of the claim. Else, there exists a component $Y\neq X$ meeting $R$ in a vertex $b$ labeled $0$. If all the remaining vertices of $R$ not in $X\cup Y$ are labeled $0$, then we have two vertices in two distinct components  and not in $X$ both labeled $0$. Else if there exists a  vertex $c$ of $R$ not in $X\cup Y$ which is labeled $1$, then replacing $u$ by its complement,  $u\dot{+}1$, the vertices $a', b', c'$ with $a':=b, b':= a, c':= c$,  satisfy the conclusion of the claim. This completes the proof of the sub claim.$\blacksquare$

Denote by $X_a$, $X_b$ and $X_c$ the components of $G$ containing $a$, $b$ and $c$, respectively. We claim that $V(R)\subseteq X_a\cup X_b\cup X_c$. Indeed, let $x\not \in  X_a\cup X_b\cup X_c$. If $\ell(x)=1$, then $x$ is adjacent to both $b$ and $c$ and therefore the graph induced on $\{a,b,c,x\}$ is a 4-cycle. Else if $\ell(x)=0$, then $x$ is adjacent to $a$ and hence $a$ has degree at least three. Hence, $x$ cannot be a vertex of $R$. This proves $V(R)\subseteq X_a\cup X_b\cup X_c$ and therefore $R$ is contained in three components of $G$. This completes the proof of Claim 1. \\
\textbf{Claim 2:} Every path of length at least five is contained in at most two components of $G$.\\
\emph{Proof of Claim $2$:} We  prove Claim 2 by induction on the length of the path contained in $G$. For the basis case we prove that every path of length five is contained in at most two components of $G$. In fact, we will prove that every path that contains the vertices $a,b,c$ defined in the Sub Claim of Claim 1 will have length at most four.\\
 Suppose that $c$ is not an end vertex of $R$. Let $c'\in R$ be the neighbor of $c$ distinct of $a$. Then $\ell(c')=1$ and $c'\in X_b$. Indeed, if $\ell(c')= 0$ then $c'\in X_c$ and since $\ell(a)=1$, $\{a, c'\}$ is an edge, which is impossible. Thus $\ell(c')=1$. If $c'\not \in X_b$, then $\{b, c’\}$ is an edge, thus $\{a, c, c', b\}$ forms a $4$-cycle. This is impossible.
 Similarly, if $b$ is not an end vertex, and if $b’$ is the  neighbor of $b$ in $R$ distinct from $a$ then $\ell(b')=1$ and $b'$ is in $X_c$. Now, if there are $a'$ and $b’$ as above, the path $c’,c, a, b,b’$  cannot be extended. Suppose not and let $x\not \in \{a,b,c,b',c'\}$ a vertex that extends the path $c’,c, a, b,b’$. If $\ell(x)=0$, then $x\in X_a$ because otherwise $a$ would have degree three in $R$ which is not possible. But then  $\{a,b,c,b',c',x\}$ induces a 6-cyle. Else if  $\ell(x)=1$, then $b$ or $c$ would have degree three which is impossible.

 Now, suppose that $b$ is an end vertex.  Then  $c$ is not an end vertex because $R$ has length five. Let $c'\in R$ be the neighbor of $c$ distinct of $a$. Then, as proved in the previous paragraph, $\ell(c')=1$ and $c'\in X_b$. Let $d$ be the neighbor of $c'$ in $R$ distinct from $c$. Then $\ell(d)=0$ and $d\in X_a$. Indeed, suppose $\ell(d)=1$, then $d$ is adjacent to $c$ or to $b$ which is impossible. It follows then that $d\in X_a$ because otherwise $d$ would be adjacent to $a$, making $a$ of degree three in $R$ which is impossible. We now prove that $\{a,b,c,c',d\}$ cannot be extended. Indeed, let $x\not \in V(R)$. If $\ell(x)=0$, then $x$ is adjacent to $c'$ or $a$, and this is impossible. Else if $\ell(x)=1$, then $x$ is adjacent to $b$ or $c$, and this is impossible.

For the inductive step, let $k>5$ be  the length of $R:=x_0,...,x_k$;  assume that every path of length $k-1$ is contained in at most two components of $G$.  By the induction hypothesis, the path $x_1,...,x_k$ is contained in at most two components $X_n$ and $X_m$ of  $G$. If $m=n$ or $x_0\in X_m\cup X_n$, then $R$ is contained in at most two components of $G$ and we are done. So we may  assume that $m\neq n$ and $x_0\not \in X_m\cup X_n$. We argue to a contradiction. We may assume that $x_1\in X_n$. Since the path $x_0,...,x_{k-1}$ has a smaller length than that of $R$, it  is contained in two  components of $G$,  namely $X_n$ and the component containing $x_0$. Thus  $V(R)\cap X_m=\{x_k\}$.  Next, we may assume without loss of generality that $\ell(x_0)=0$ (otherwise, replace $u$ by $u\dot{+} 1$). Then $\ell(x_k)=0$ (this is because $\{x_0,x_k\}$ is not an edge and $x_0$ and $x_k$ are in different components) and $\ell(x_{1})=1$ (this is because $\{x_0,x_{1}\}$ is an edge and $x_0$ and $x_{1}$ are in different components). From $\ell(x_{1})=1$ and $\ell (x_k)=0$ we get that $\{x_1,x_k\}$ is an edge hence $k=2$, a contradiction. This completes the proof of Claim 2.\\
We should mention that the bound in Claim 2 is sharp. For an example see Figure \ref{fig3}.\\

\begin{figure}[h]
\begin{center}
\includegraphics[width=100pt]{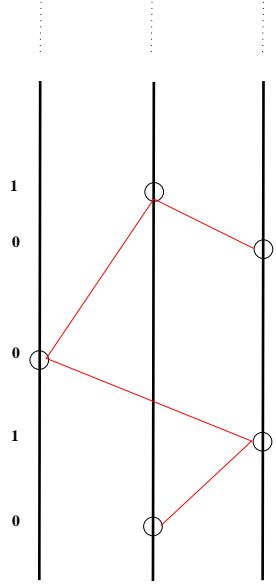}
\end{center}
\caption{An example of a path of length four contained in three components in the case $\star$ is the Boolean sum. The thick vertical lines represent components of $\widehat G_{(u, \star)}$.}
\label{fig3}
\end{figure}


\noindent \textbf{Claim 3:} If $R$ is contained in two distinct components $X_m$ and $X_n$ of $G$ such that $|V(R)\cap X_m|=1$ or $|V(R)\cap X_n|=1$, then the length of $R$ is less than $\varphi(u)$.\\
\emph{Proof of Claim 3:} Let $R$ be a path that is contained in $G$ so that $|V(R)\cap X_m|=1$. Let $x$ be such that $V(R)\cap X_m:=\{x\}$. Without loss of generality we may assume that $\ell(x)=0$ (otherwise, replace $u$ by $u\dot{+} 1$). We suppose that the length of $R$ is at least $\varphi(u)$ and we argue to a contradiction. Since the degree of $x$ in $R$ is at most two we infer that the graph induced by $V(R)\setminus \{x\}$ on $X_m$ is the direct sum of at most two paths $R_1$ and $R_2$ where $R_1$ or $R_2$ could be empty. Hence $\vert V(R) \vert = \vert V(R_1)\vert +\vert V(R)\vert +1$. Since the length of $R$ is at least $\varphi(u):= 2(l(u)+1)+1$,  we infer that $R_1$ or $R_2$ has length at least $l(u)+1$, say $R_1$ has length at least $l(u)+1$. From $x$ is adjacent to exactly one vertex $y$ of $R_1$ and $\ell(x)=0$ we deduce that $\ell(y)=1$ and all other vertices $z$ of $R_1$ verify $\ell(z)=0$. Hence $u$ has a factor of zero's of cardinality $l(u)+1$ which is impossible. This completes the proof of Claim 3.\\
\textbf{Claim 4:} If $R$ is contained in two distinct components $X_m$ and $X_n$ of $G$ such that  $|V(R)\cap X_m|\geq 2$ and $|V(R)\cap X_n|\geq 2$, then the length of $R$ is at most five.\\
\emph{Proof of Claim 4:} Since $R$ is connected we infer that, up to a permutation of $m$ and $n$, there are vertices $x$ and $y$ such that $x\in V(R)\cap X_m$, $\ell(x)=0$, and $y\in V(R)\cap X_n$, $\ell(y)=1$. From $\ell(y)=1$ we deduce that $|V_0(R)\cap X_m|\leq 2$. Now suppose for a contradiction that all vertices of $V(R)\cap X_m$ are labeled $0$. From $|V_0(R)\cap X_m|\leq 2$ and our assumption $|V(R)\cap X_m|\geq 2$ we deduce that $|V(R)\cap X_m|= 2$. In particular, $y$ has degree two in $R$. Therefore,  no vertex $v \in V(R)\cap X_n$  distinct from $y$ is joined to $y$ and furthermore can be labeled $1$ because otherwise the graph induced on $\{y,v\}\cup (V(R)\cap X_m)$ is a 4-cycle. But then there is no subpath of $R$ that connects   $y$ and $v$ (such a path will have to contain some vertex labelled $1$ and there are none). This proves that not all vertices of $V(R)\cap X_m$ can be labeled $0$. By symmetry we obtain that not all vertices of $V(R)\cap X_n$ can be labeled $1$. Let $t\in V(R)\cap X_m$ be labeled $1$ and let $z\in V(R)\cap X_n$ be labeled $0$.

Next,  we prove that $|V_0(R)|\leq 3$. Suppose for a contradiction that there are distinct vertices $r,s\in V(R)$ labeled $0$ and such that $\{r,s\}\cap \{x,z\}=\varnothing$. Then $\{r,s\}\nsubseteq X_m$ and $\{r,s\}\nsubseteq X_n$ because otherwise $y$ or $t$ would have degree at least three which is not possible. Say $r\in X_m$ and $s\in X_n$ and note that $t$ and $y$ have degree two. We prove  that $V(R)=\{r,s,t,x,y,z\}$. This  leads to a contradiction since $\{r,y,x\}$ and $\{s,t,z\}$ are then nontrivial connected components of $R$. For the proof, suppose  that $q\not \in \{r,s,t,x,y,z\}$ is a vertex of $G$. If $\ell(q)=0$, then $q$ is adjacent to $t$ if $q\not \in X_m$ or $q$ is adjacent to $y$ if $q\not \in X_n$. Hence, $t$ or $y$ have degree three which is impossible. Else if $\ell(q)=1$, then $q$ is adjacent to $x$ and $r$ if $q\not \in X_m$ and therefore $\{q,x,y,r\}$ induces a $4$-cycle, or $q$ is adjacent to $z$ and $s$ if $q\not \in X_n$ and therefore $\{q,t,z,s\}$ induces a $4$-cycle. In both cases we obtain a contradiction. This proves our claim $V(R)=\{r,s,t,x,y,z\}$. Similarly, we prove that $|V_1(R)|\leq 3$ and therefore $|V(R)|\leq 6$. This proves that the length of $R$ is at most five, as claimed. This completes the proof of Claim 4.

The conclusion of the lemma in the case $\star$ is the Boolean sum is obtained as follows. Let $R$ be an induced path of length at least  $\varphi(u)$ in $G$. It follows from Claim 2 and $\varphi(u)\geq 6$ that $R$ is contained in at most two components of $G$. It follows from Claims 3 and 4 that $R$ cannot be included in two distinct components of $G$. Hence, $R$ is contained in one component of $G$ as required.

\item $\star$ is the second projection, that is $i\star j=j$ for all $i,j$.
We note that since $R$ is connected, there exist $n'<n$, $x'\in V(R)\cap X_{n'}$, $x\in V(R)\cap X_{n}$  such that $\{x',x\}\in E(R)$.
Since $\{x',x\}\in E(R)$,  $\ell(x)=1$.  Since $\star$ is the second projection, $x$ is adjacent to all vertices of $V(R)\setminus X_n$. Hence, $|V(R)\setminus X_n|\leq 2$ because otherwise the degree of $x$ in $R$ would be at least three which is not possible. \\
\noindent \textbf{Claim 1:} $|V(R)\setminus X_n|=1$. Hence,  $V(R)\setminus X_n=\{x'\}$.\\
Suppose that  $|V(R)\setminus X_n|=2$.  If $y\in  V(R)\cap X_n$  distinct from $x$,  then $y$ cannot  be adjacent to $x$ (this is because $x$ has degree two in $R$).  Furthermore, $\ell(y)=0$ because otherwise $y$ would be adjacent to all vertices of $V(R)\setminus X_n$ and hence $R$ would have a $4$-cycle which is not possible. From that,  $y$ is not adjacent to any vertex of $R$ which is impossible. Consequently, $|V(R)\cap X_n|=1$. It follows that  $|V(R)|=3$.
Due to $\varphi (u) \geq 5$, we have $\vert V(R)\vert \geq 6$, a contradiction. This completes the proof of  Claim 1.

\emph{Note that since every vertex of $V(R)\setminus X_n$ is adjacent to every vertex of $V_1(R)\cap X_n$ we must have $|V_1(R)\cap X_n|\leq 2$.}

\noindent \textbf{Case 1.} $|V_1(R)\cap X_n|=1$.\\
In this case, the  length of $R$ is at most $l_0(u)+1$.\\
Indeed, we have $V_1(R)\cap X_n=\{x\}$. Hence $x$ is the only neighbour of $x'$. Thus
$x'$ is an end vertex of  $R$, the remaining vertices of $R$ are in $X_n$,  all vertices distinct from $x$ being labelled $0$. So the length of $R$ is $l_0+1$ as claimed.
Since $\varphi(u)>l_0(u)+1$ this case is not possible.\\
\noindent \textbf{Case  2.}  $|V_1(R)\cap X_n|=2$.\\
In this case,   the  length of $R$ is at most $2(l_0(u)+1)$.\\
Indeed, the vertex $x'$ is linked to the two vertices of $V_1(R)$, thus, $R\setminus \{x'\}$ has two connected components, each one made of a path in $X_n$ labelled $0$ except an end vertex belonging to $V_1(R)$. Thus the length of $R$ is $2(\ell_0(u)+1)$.
Since $\varphi(u)>2(l_0(u)+1)$ this case is not possible.
\end{enumerate}
\end{proof}

This proof is quite different from the original one. We thank one of the referees of this paper for suggesting an other approach that eventually lead us to the proof given above.

\begin{remark}Recall that $G\cdot_{\star} \omega^d$ and $G\cdot_{\star} \omega$ have the same age (Corollary \ref{cor:zeta-omega}). If the operation $\star$ is symmetric, then  $G\cdot_{\star} \omega^d$ and $G\cdot_{\star} \omega$ are equimorphic and the corresponding $\widehat Q_{(u, \star)}$'s are path-minimal. On the other hand, if $\star$ is the second projection the corresponding $\widehat Q_{(u, \star)}$'s are path-minimal but not equimorpic.
\end{remark}

%
\subsubsection{Incomparable ages}

 We complete the proof of  Theorem \ref{thm:uncoutable ages} by proving:

\begin{proposition}\label {propdistinctages}
Let $\star$ be the Boolean sum or the second projection on $\{0,1\}$. If $u$ and $u'$ are  two uniformly recurrent words which are not constant and  the factor sets $\fac (u)$ and $\fac (u')$ are nonequivalent in the sense of Definition \ref{def:equivalence} then $\age(\widehat G_{(u, \star)})$ and $\age(\widehat G_{(u', \star)})$ are incomparable w.r.t. set inclusion.
\end{proposition}

Let   $G_{(u, \star)}$ be the lexicographical sum of copies of $P_u$ and $G:= \widehat G_{(u, \star)}$ be the resulting graph on $\NN\times \NN$. We associate to this graph $G$ a  set $\mathcal A_G(u)$ of  words over the alphabet $\{0,1\}$ defined as follows. A word $w:=w_0\cdots w_{n}\cdots  w_{m-1}$ belongs to $\mathcal A_G$ if  there is a path $(j,v_0), \dots,  (j, v_{m-1})$ in $\{j\}\times \NN$ and some $(i, v) \in \NN\times \NN$ with   $j\not =i$  such that  $w_n=1$ if and only if  the pair $\{(j, v_n), (i,v))\}$ is an edge of $G$ for all $n<m$.

Taking account of the definition of $G_{(u, \star)}$ the condition on $w$ amounts to say that either $w_n =  u(v)\star u(v_n)$ for all $n<m$ if $j<i$ or $w_n= u(v_n)\star u(v)$ if $i<j$.

Let $u:= (u(n))_{n\in I}$ be a word on a general alphabet $L$, the index set $I$ being an interval of $\ZZ$. For $i\in L$ set $i\star u:= (i\star u(n))_{n\in I}$. For a set $U$ of words set $i\star U:= \{i\star u: u\in U\}$ and if $X$ is a subset of $L$ set $X\star U:= \bigcup_{i\in X} i\star U$. Similarly define $u\star i$, $U\star i$, $U\star X$.

With these notations we have immediately:
\begin{lemma}\label{lem:keyword} $\mathcal A_G(u)$ is the set union of $L\star \fac (u)$, $\fac(u)\star L$ and their duals.
\end{lemma}

Specializing to $L=\{0,1\}$, if $\star$ is the Boolean sum then $L\star \fac (u)= \fac(u)\star L= \fac(u)\cup \fac (u {\dot+} {\bf 1})$, while $L\star \fac (u)= \fac(u)$ and $\fac (u)\star L= \fac({\bf 0})\cup \fac ({\bf {1}})$ if $\star$ is the second projection.\\

\noindent {\bf Proof of Proposition \ref {propdistinctages}.} The proof goes in two steps.

\begin{enumerate}
\item If $\age(\widehat G_{(u, \star)}) \subseteq \age(\widehat G_{(u', \star)})$ then $\mathcal A_G(u)\subseteq \mathcal A_G(u')$.  \label{item:step1}
\item If $\mathcal A_G(u)\subseteq \mathcal A_G(u')$ then $\fac (u)$ and $\fac (u')$ are equivalent. \label{item:step2}
\end{enumerate}

The proof of Item (\ref{item:step1}) is easy. Let $w\in \mathcal A_G(u)$. Let $H$ be a path in a component and a vertex $(i, v)$ in $G_{(u, \star)}$ witnessing this fact. Since the component are infinite  this path is included in a path $H'$ long enough so that Lemma \ref{lem:longenough} ensures that
 in any  embedding of $\widehat G_{(u, \star)} {\restriction H'\cup \{(i,v)\}}$ in $\widehat G_{(u', \star)}$ the image of $H'$ is mapped into  a component. Since necessarily $(i, v)$ is mapped into another component, from this embedding we get that $w\in \mathcal A_G(u')$.

The proof of Item (\ref{item:step2}) goes as follows. Suppose that   $\mathcal A_G(u)\subseteq  \mathcal A_G'(u')$. Then, by Lemma \ref{lem:keyword},  $\fac (u)$ is included in the set union of $L\star \fac (u')$, $\fac(u')\star L$ and their dual. Since $\fac(u)$ is an ideal included in  a finite union of initial segments,  it is included in one of them. Hence it is included in a set of the form $\fac(v)$ equivalent to  $\fac(u')$ or equal to $\fac(\bf{0})$ or to  $\fac(\bf{1})$. According to Theorem \ref{unif-recurrent}, since $v$ is uniformly recurrent, $\fac(v)$ is J\'onsson.  Hence $\fac (v)= \fac(u)$. Since $u$ is non constant,  $v$ is distinct from $\bf {0}$ and $\bf{1}$, hence $\fac (v)$ is equivalent to $\fac(u')$. \hfill $\Box$

%
%
%
%

\subsection{Primality and path-minimality}

Let $G:=(V,E)$ be a graph. If $G$ is path-minimal then, by definition,  $G$ is embeddable  in $G_{\restriction V'}$ provided that this restriction contains  finite induced paths of unbounded length. And,  in this case, $G_{\restriction V'}$ is path-minimal too. In some cases, as presented in Theorem \ref{claim1} and \ref{thm-prime-path-minimal},  $G$ is prime; in other cases, e.g. in Theorem \ref{thm-nonprime-path-minimal}, no graph equimorphic to $G$ is prime.

\begin{theorem}\label{claim1} The graphs $\widehat Q_k$, for  $k\geq 2$, are prime.
\end{theorem}

\begin{proof}
Let $M$ be a nontrivial module in $\widehat Q_k$.
\begin{itemize}

\item  $M$ meets each  component in at most one vertex.

Indeed, suppose that $M$ contains at least two vertices, say $a$ and $b$, of some component $X_i$ of $G$. Then, it contains $X_i$. Indeed, first, it contains the vertices  of $X_i$ exterior to the path joining $a$ and $b$; since $X_i:= \{(i,j): j<i+5\}$ it contains the vertices on that path. Next, we prove that $M = X$. For that we prove that $M$ contains $X_k$ for all $k\neq i$. Let $a := ( i, 0 )$, $b := ( i, 1 )$, $c := ( k, 0 )$, $d := ( k, 1 )$. By construction $d\sim b$ and $d\sim a$. Hence by the same token, $c,d \in M$. Since $M$ is a module containing $a$ and $b$ it contains $c$. Similarly $c\nsim a$ and $c\sim b$. Hence by the same token, $d\in M$. Since $M$ contains two elements of $X_k$ it contains all of it. Thus $M =X$. Contradicting the fact that $M$ is nontrivial.

\item $M$ is a singleton.

Indeed, suppose that $M$ contains at least two elements $a:= (i, n)$ and $b:= (j, m)$. According to the previous item, $i\not =j$. Without loss of generality we may suppose that $i>j$, so that $\vert X_i\vert \geq 6$. We claim that there is some $a':= (i, n')$ distinct of $a$ such that $a'$ is not linked to $a$ and $b$ in the same way. Since $M$ is a module, this implies that $a$ belongs to $M$ and this contradicts the previous item. In order to prove our claim, set $v:=1$ if $n\not \equiv_k m$ and $v=0$ otherwise.  Suppose that $v= 0$, that is $a\not \sim b$. Since $\vert X_i\vert \geq 6$, one of the sets  $\{n+1, n+2, n+3\}$ and  $\{n-1, n-2, n-3\}$ is included in $X_i$.  Suppose that $\{n+1, n+2, n+3\}\subseteq X_i$. If $k>2$  then  for $n':=n+2$, $a'\in X_i$ and $a'\not\sim a$ but $a'\sim b$. If $k=2$ then the same conclusion holds for $n':=  n+3$. Suppose that $v=1$, that is $a\sim b$.   If $n+1\equiv_k m$, set $n':= n+1$. If not, then either $n+2\not \equiv_k m$ in which case set $n':= n+2$, or $n+2 \equiv_k m$, in which case $n+3\not \equiv_k m$ and we may set $n':=  n+3$.
\end{itemize}
\end{proof}

\begin{remark} $\widehat Q_k$, $k\geq 2$, is not minimal prime in the sense of Definition \ref{minimality-pou-zag}.
\end{remark}
Indeed, with similar arguments as in the proof of Theorem \ref{claim1}, the restriction of $\widehat Q_k$ to $\NN\times \{0,1,2,3,4\}$ is also prime and hence $\widehat Q_k$ is not minimal prime in the sense of Definition \ref{minimality-pou-zag}.

\begin{theorem}\label{thm-prime-path-minimal} There are $2^{\aleph_0}$ path-minimal prime graphs whose ages are pairwise incomparable and wqo.
\end{theorem}
\begin{proof}Let $\star$ be the Boolean sum. Let $u$ be a uniformly recurrent and non constant word, let  $G_{(u, \star)}$  be the labelled graph on $\NN\times \NN$ associated to $u$, let $\ell$ be the labelling  and  $G$ be the restriction of  $\widehat G_{(u, \star)}$ to the set $X:= \{ (n, m): m\leq n+ \varphi (u)\}$. Containing only paths of finite unbounded  length, this graph is path-minimal. We show in an almost identical  way as in Theorem \ref{claim1} that this graph is prime.

Let $M$ be a nontrivial module in $G$.

\begin{itemize}
\item  $M$ meets each  component in at most one vertex.
Indeed, suppose that $M$ contains at least  two vertices, say $a$ and $b$,  of some component $X_i$ of $\widehat Q_k$. Then it contains $X_i$. Indeed, note first that it contains  the vertices in $X_i$ exterior to the path joining $a$ to $b$, next note that  it contains the vertices of $X_i$  on that path (since $X_i:= \{(i,j): j<i+5\}$). Next we prove that $M=X$. Let $c:= (j, k)\not \in X_i$, that is $j\not =i$. Since $\vert X_i\vert\geq  \varphi(u)$, $X_i$ contains $a:= (i, n)$ and $b:= (i, m)$ with $u(n)\not =u(m)$, hence $\ell(a)\dot{+}\ell(c)= u(n)\dot{+}u(k)\not =u(m)\dot{+}u(k)=\ell(b)\dot{+}\ell(c)$. That means that $c$ is not linked to $a$ in the same way as to $b$. Since $M$ is a module, it contains $c$. Thus $M=X$. A contradiction.

\item $M$ is a singleton.

Indeed, suppose that $M$ contains at least two elements $a:= (i, n)$ and $b:= (j,m)$. According to the previous item, $i\not =j$. We claim that there is some $c:= (i, k)$ distinct of $a$ such that $c$ is not linked to $a$ and $b$ in the same way. Since $M$ is a module, this implies that $c$ belongs to $M$ and this contradicts the previous item. In order to prove our claim, set $v:= \ell(a) \star \ell (b)$. Suppose that $v= 0$, that is $a$ not linked to $b$. Since $\vert X_i\vert \geq \varphi(u)$ we may find $k$ with $\vert k-i\vert \geq 2$ and $\ell(c)\not = \ell (a)$, proving our claim in that case. We do similarly if $v=1$.
\end{itemize}
\end{proof}

\begin{theorem}\label{thm-nonprime-path-minimal}There are $2^{\aleph_0}$ path-minimal non prime graphs whose ages are pairwise incomparable and wqo and such that any graph equimorphic to these graphs is also nonprime.
\end{theorem}
\begin{proof}
Let $\star$ to be the second projection, i.e.,    $i\star j=j$ for all $i,j\in \{0,1\}$. Consider the graph  $\widehat G_{(u, \star)}$ for a uniformly recurrent word $u$. For $i\in \NN$, set $X_i:= \{i\}\times P_u$. Let $k\in \NN$. We claim that $M_k:=\cup_{j=0}^{k} X_i$ is a module in $\widehat G_{(u, \star)}$. Indeed, if $x\not \in M_k$, then either $\ell(x)=0$, in which case $x$ is not adjacent to any vertex of $M_k$, or $\ell(x)=1$, in which case $x$ is adjacent to every vertex of $M$. Now let $G$ be a graph equimorphic to $\widehat Q_{(u, \star)}$. Clearly the trace of $M_k$ on $G$ is a module in $G$ and if is $k$ large enough,  $M_k$ is nontrivial proving that $G$ is not prime.  According to Proposition \ref {prop:pathminimal} this graph is path-minimal. According to  Proposition \ref {propdistinctages} (applied to the second projection) there are  $2^{\aleph_0}$ graphs of that form with pairwise incomparable ages.
\end{proof}
\begin{remark} We get the same conclusion by replacing the  second projection by the first. Indeed, in that case, the sets $\NN \setminus M_k$ are   modules of  $\widehat G_{(u, \star)}$.
\end{remark}

\section{Proof of Theorem \ref{thm:incomparability graph}}

Let $G:=(V,E)$ be a graph and let $F\subseteq V$. We define a binary relation $\equiv_F$ on $V\setminus F$ as follows. For $v,v'\in V\setminus F$ set
\[v\equiv_F v' \mbox{ if for all } f\in F \quad (\{f,v\}\in E \Leftrightarrow \{f,v'\}\in E).\]
It is easily seen that $\equiv_F$ defines an equivalence relation on $V\setminus F$. Let $\tau_G(F)$ be the partition of $V\setminus F$ induced by $\equiv_F$. Each equivalence class $B$ is completely determined by a subset of $F$, namely, the subset of $f\in F$ such that $\{f,v\}\in E$ for every $v\in B$. Hence, the number of equivalence classes of $\equiv_F$ is at most $2^{|F|}$.

\begin{lemma}\label{lem1}Let $G:=(V,E)$ be a graph, $F\subseteq V$ and $B\in \tau_G(F)$. Then the equivalence $\equiv_B$ induces a partition of $F$ in at most two subsets $A$ and $A'$. In particular, $G_{\restriction A\cup B}$ is the direct or complete sum of $G_{\restriction A}$ and $G_{\restriction B}$.
\end{lemma}
\begin{proof}Let $A$ be the set of $u\in F$ such that $\{u,x\}\in E$ for some $x\in B$. Let $u'\in A$. Since $B$ is an equivalence class of $\equiv_F$ we have $\{u',x'\}\in E$ for all $x'\in B$. Furthermore, if $u'\in F$ we have $u\equiv_B u'$ if $u'\in A$. Hence $A$ is an element of the partition induced by $\equiv_B$ and its complement in $F$ is also an element of the partition induced by $\equiv_B$. This proves the first assumption of the lemma. The rest is a consequence.
\end{proof}

Paths in the incomparability graph of a poset have a special form given the following lemma which is an improvement of I.2.2 Lemme, p.5 of \cite{pouzet78}.

\begin{lemma}\label{lem:inducedpath} Let $x,y$ be two vertices of a poset $P$ with $x<y$.  If $x_0, \dots, x_n$ is an induced path from $x$ to $y$ in the incomparability graph of $P$ then   $x_i< x_j$ for all $j\geq i+2$.
\end{lemma}
\begin{proof}
We proceed by induction on $n$. If $n\leq 2$ the property holds trivially. Suppose $n\geq 3$. Taking  out $x_0$, induction applies to $x_1, \dots, x_n$. Similarly, taking out $x_n$, induction applies to $x_0, \dots, x_{n-1}$. Since the path from $x_0$ to $x_n$ is induced, $x_0$ is comparable to every $x_j$ with $j\geq 2$ and $x_n$ is comparable to every $x_j$ with $j<n-1$. In particular, since $n\geq 3$, $x_0$ is comparable to $x_{n-1}$. Necessarily, $x_0< x_{n-1}$. Otherwise, $x_{n-1}<x_0$  and then by transitivity $x_{n-1}<x_{n} $ which is impossible since $\{x_{n-1}, x_{n}\} $  is an edge of the incomparability graph. Thus, we may apply induction to the path  $x_0, \dots,  x_{n-1}$ and  get $x_0<x_j$ for every $j>2$. Similarly, we get $x_1<x_n$ and via the induction applied to the path from $x_1$ to $x_n$, we obtain $x_j< x_n$ for $j<n-1$. The stated result follows.
\end{proof}

If $P:=(V,\leq)$ is a poset, a subset $X$ of $V$ is \emph{order convex} or \emph{convex} if for all $x,y\in X$, $[x,y]:=\{z : x\leq z\leq y\}\subseteq X$.

\begin{lemma}\label{lem3}Let $P:=(V,\leq)$ be a poset. If a convex subset $C$ of $V$ contains the end  vertices  of an induced path of length $n+3$, then it contains an induced subpath of length $n$.
\end{lemma}
\begin{proof}Let $x_0,...,x_{n+3}$ be an induced path of length $n+3$ with $x_0,x_{n+3}\in C$. It follows from Lemma \ref{lem:inducedpath} that $x_0\leq x_i\leq x_{n+1}$ if $2\leq i\leq n+3$. Hence the path $x_2,...,x_{n+1}$ is in $C$.
\end{proof}

Let $P:=(V,\leq)$ be a poset. For $v\in V$, denote by $D(v):=\{u\in V : u<v\}$, $U(v):=\{u\in V : u>v\}$ and $Inc(v):=\{u\in V: u \mbox{ is incomparable to } v\}$.

\begin{lemma}\label{lem4}If $G$ is the incomparability graph of a poset $P=(V,\leq)$, then for each subset $F$ of $V$, each element of $\tau_G(F)$ is convex.
\end{lemma}
\begin{proof}Observe first that for each $u\in V$, the sets $U(v)$ and $D(v)$ and $Inc(v)$ are convex. Next, observe that the intersection of convex subsets is convex. To conclude that an element of $\tau_G(F)$ is convex it is enough to prove that it is the intersection of convex sets. For that, let $v\not \in F$. For each $y\in F$ set $V_v(y)$ to be $Inc(y)$ if $v\|y$, $U(y)$ if $y<v$, or $D(y)$ if $v<y$. We claim that the equivalence class containing $v$ is $B:=\bigcap_{y\in F}V_v(y)$. Indeed, by definition $v\in B$ and $B$ is convex since intersection of convex sets. Next, if $v'\equiv_F v$, then $v'\in B$. Finally, if $v'\not \in B$, then for some $y\in F$, $v'\not \in V_v(y)$ and thus $v'\nequiv_F v$.
\end{proof}

\begin{lemma}\label{lem5}Let $n,k$ be nonnegative integers and $f_k(n)=k\cdot (n+4)$. If the incomparability graph of a poset $P$ contains an induced path $C$ of length at least $f_k(n)$, then for every covering of $P$ by at most $k$ convex sets, one convex set contains a subpath  of $C$ of length  $n$.
\end{lemma}
\begin{proof}We start with the following claim.\\
\begin{claim} \label{claim7} If the vertices of a path of length $m$ are coloured with $k$ colors, then there are two vertices of the same color at distance at least $\lfloor \frac{m}{k}\rfloor $.
\end{claim}
\noindent {\bf Proof of  Claim \ref{claim7}} Let $v_0,v_1,...,v_m$ be a path, $l:=\lfloor \frac{m}{k}\rfloor $ and $w_0:=v_0,w_1:=v_l,w_2:=v_{2l},...,w_k:=v_{kl}$. The colors of these $k+1$ vertices cannot be all different, so two vertices $w_i,w_j$ with $i<j$ have the same color. Their distance is at least $l$.$\Box$\\
If we have a path of length at least $f_k(n)=k\cdot (n+4)$, there are two vertices $x,y$ on this path at distance at least $n+4$. The induced path is $x=x_i,x_{i+1},x_{i+2},...,x_{i+n+2}, x_{i+n+3},x_{i+n+4}$. According to the proof of Lemma \ref{lem3}, all vertices $x_{i+2},...,x_{i+n+2}$ are in the convex hull of $\{x_i,x_{i+n+4}\}$.
\end{proof}

\begin{lemma}\label{lem6}Let $P:=(V,\leq)$ be a poset. If $Inc(P)$ contains induced finite paths of unbounded  length, then for every integer $n$ there is an induced path $C$ of length $n$ and a subset $V'$ of $V\setminus V(C)$ such that:
\begin{enumerate}[$(1)$]
  \item $V'$ is convex and contains induced paths of $Inc(P)$ of unbounded finite lengths.
  \item Either all vertices $x\in V(C)$ and $y\in V'$ are adjacent, or nonadjacent in $Inc(P)$.
\end{enumerate}
\end{lemma}
\begin{proof}Since $Inc(P)$ contains induced paths of unbounded length, it contains a path $C'$ of length $n'\geq 2\cdot (n+4)$. Let $F:=V(C')$. Since $Inc(P)$ contains induced paths of unbounded lengths and $F$ is finite, the same holds for $V\setminus F$. The partition of  $V\setminus F$ in equivalence classes with respect to $\equiv_F$ has at most $2^{n'+1}$ equivalence classes. According to Lemma \ref{lem4} each block is convex. According to Lemma \ref{lem5} some of these equivalence classes contain paths of unbounded length. According to Lemma \ref{lem1}, the equivalence $\equiv_{V'}$ on $F$ induces a partition in at most two sets. According to Lemma \ref{lem5}, one of these two sets contains a subpath $C$ of $C'$ of length $n$. With $C',V'$ we get the required conclusion.
\end{proof}
We prove Theorem \ref{thm:incomparability graph} as follows. With Lemma \ref{lem6} we build a sequence $(C'_n,V'_n)_{n\in \NN}$ where $C'_n$ is an induced path of length $n$ of $V'_{n-1}$, with $V'_{-1}=V$ and such that $V'_n$ is convex and all its elements are either all adjacent or nonadjacent to $C'_n$. This does not prevent the situation that at some stage $C'_n$ is nonadjacent to all elements of $V'_n$ and some other $C'_n$  adjacent to all. But at least either there will be infinitely many in the first case, in which case we obtain a direct sum of paths of arbitrarily large lengths, or not, in which case we obtain a complete sum.

\section{Proof of Theorem \ref{thm:isometric}}
The result is a consequence of the following lemma.
\begin{lemma}The supremum of the diameters of the connected components of a graph $G$ is infinite if and only if $G$ embeds a direct sum of isometric paths of arbitrarily large lengths.
\end{lemma}
\begin{proof}Let us associate to each connected component $C$ of $G$ the supremum of the lengths of the isometric paths in $C$. If all these suprema are finite we are done (in each connected component we can select a path of maximal length). Otherwise, for some connected component the supremum is infinite. We can assume that $ G $ is the graph induced on this connected component. Let $ V $ be its set of vertices. If $ X $ is a subset of $ V $ it has a diameter $ \delta (X) $.  If this diameter is finite, the diameter of $B_G(X, 1)$ is also finite (at most $ \delta (X) + 2 $), and hence $B_G(X, m) $ is finite for any every $ m \geq 1 $. The diameter of $V \setminus B_G(X, m)$ is infinite (indeed if $U, W$ are two subsets the diameter of $U \cup W$ is at most $ \delta(U) + \delta(W) + d(U, W) $). Choose an integer $ m $ arbitrarily.
Take $ k <m $ and two elements $x, y$ in $V \setminus B_G(X, m)$ at distance $ k$ and a path $R$ of length $ k $. If a vertex of this path is in $B_G(X, 1) $ then $ x $ is in $B_G(X, k + 1) $, which is impossible if $ k \leq m-1 $. Therefore, the isometric path is in the complement of $B_G (X, 1) $ and therefore it has no edge in common with the other paths already constructed.
\end{proof}

\acknowledgements
\label{sec:ack} Results of this paper have been presented by the second author to ALGOS 2020 Algebras, Graphs and Ordered Sets, Loria, Nancy, August 26th to 28th 2020. The authors are pleased to thank the organizers for their warm welcome.\\
The authors are grateful to two  anonymous referees for their careful reading of the manuscript and their many  comments and suggestions.

\bibliographystyle{abbrvnat}
\bibliography{long-paths}

\end{document}